\documentclass[10pt]{article}

\usepackage{amsmath}
\usepackage{amsthm}
\usepackage[all]{xy}
\usepackage{latexsym}

\newcommand{\DDelta}{\mathbf{\Delta}}
\newcommand{\BB}{\mathbf{B}}
\newcommand{\FF}{\mathbf{F}}
\newcommand{\QQ}{\mathbf{Q}}
\newcommand{\ZZ}{\mathbf{Z}}
\newcommand{\CALA}{\mathcal{A}}
\newcommand{\CALM}{\mathcal{M}}
\newcommand{\CALS}{\mathcal{S}}
\newcommand{\one}{\mathbf 1}
\newcommand{\norm}[1]{{\mathopen\parallel}{#1}{\mathclose\parallel}}

\newtheorem{theorem}{Theorem}
\newtheorem{proposition}{Proposition}
\theoremstyle{definition}
\newtheorem{definition}{Definition}
\newtheorem{remark}{Remark}

\begin{document}

\title{The algebra of rack and quandle cohomology.}
\author{F.J.-B.J.~Clauwens}

\maketitle

This paper presents the first complete calculation
of the cohomology of any nontrivial quandle,
establishing that this cohomology exhibits a very rich 
and interesting algebraic structure.

Rack and quandle cohomology have been applied
in recent years to attack a number of problems in the theory of knots
and their generalizations like virtual knots and higher dimensional knots.
An example of this is estimating the minimal number of triple points
of surface knots \cite{hatakenaka}.
The theoretical importance of rack cohomology
is exemplified by a theorem \cite{fennrs2} identifying the homotopy groups
of a rack space (see $\S3$ ) with a group of bordism classes of high dimensional knots.
There are also relations with other fields, like the study
of solutions of the Yang-Baxter equations.

\section{Introduction.}

\subsection{Definition and examples.}
\begin{definition}
A \emph{quandle} is a set $X$ with  binary operation $(a,b)\mapsto a*b$ such that
\begin{enumerate}
\item	For any $a\in X$ we have $a*a=a$.
\item For any $a,b\in X$ there is a unique $c\in X$ such that $a=c*b$.
\item For any $a,b,c\in X$ we have $(a*b)*c=(a*c)*(b*c)$.
\end{enumerate}
A \emph{rack} is a set with a binary operation which satisfies (2) and (3).
A homomorphism  $f\colon X\to Y$ between racks is a map such that
$f(a*b)=f(a)*f(b)$ for all $a,b\in X$.
\end{definition}
\begin{remark}
Some authors, for example \cite{andrug3} and \cite{armstrong1}, write $b*a$ where we 
and most others write $a*b$.
\end{remark}

The following are typical examples of quandles.
\begin{itemize}
\item
Any group $G$ gives rise to a quandle $X=Conj(G)$ operation
$a*b=b^{-1}ab$.
This is the \emph{conjugation} quandle of $G$.
More generally any conjugation invariant subset of $G$ gives rise to a quandle.
For example the reflections in the dihedral group $D_n$
yield the \emph{dihedral quandle} $R_n$,
which will be studied in this paper.
\item
An abelian group $M$ with an automorphism $T$ gives rise to a quandle
$X=Alex(M,T)$ by the formula
\begin{equation}
a*b=Ta+(1-T)b
\end{equation}
This is the \emph{Alexander quandle} of $(M,T)$.
For example $Alex(\ZZ/(n),-1)$ is just $R_n$.
\item
Any oriented classical knot or link diagram $K$
gives rise to a quandle called its \emph{fundamental quandle}.
The axioms for a quandle corrsepond to Reidemeister moves of type I,II,III respectively
(see \cite{fennr} and \cite{kauffman1} and \cite{carterks5}).
A Fox n-coloring is just a quandle homomorphism 
from $K$ to $R_n$.
See \cite{joyce1} and \cite{fennr} and \cite{rourkes1}
for increasingly strong theorems about the degree to which
the fundamental quandle determines a knot.
\item
Simple curves on a surface give rise to a quandle using Dehn twists.
See \cite{yetter3} and \cite{yetter4}.
\item
Any set $S$ gives rise to a quandle 
by the formula $a*b=a$ for $a,b\in S$.
This is called the \emph{trivial} quandle of $S$.
\item
One can construct a quandle by  taking the disjoint union $\ZZ/(k)\cup \ZZ/(m)$ 
and defining $a*b=a$ if $a$ and $b$ are in the same part and $a*b=a+1$ if they are not.
\end{itemize}
The last example suggests that quandles can be glued together in disturbingly many ways.
For this reason we concentrate in this paper on connected quandles
(see next section for the definition) which seems to be a class more 
amenable to understanding.

\subsection{Rack and quandle homology.\label{hom}}
In \cite{fennrs2} a homology theory for racks was defined,
which was modified in  \cite{carterjkls2} to yield a homology theory for quandles.
For a rack $X$ let $C^R_n(X)$ be the free abelian group generated by $X^n$.
Define a map $\partial\colon C^R_n(X)\to C^R_{n-1}(X)$ as follows:
\begin{equation}
\begin{split}
&\partial^0_i(x_1,\dots,x_n)=(x_1,\dots,x_{i-1},x_{i+1},\dots,x_n)\\
&\partial^1_i(x_1,\dots,x_n)=(x_1*x_i,\dots,x_{i-1}*x_i,x_{i+1},\dots,x_n)\\
&\partial^0=\sum_{i=1}^n(-1)^i\partial^0_i,\qquad
\partial^1=\sum_{i=1}^n(-1)^i\partial^1_i,\qquad
\partial=\partial^0-\partial^1
\end{split}
\end{equation}
It can easily be checked that
\begin{equation}
\partial^0\partial^0=0,\qquad
\partial^1\partial^1=0,\qquad
\partial^0\partial^1+\partial^1\partial^0=0
\end{equation}
Therefore $\{C^R_n(X),\partial\}$ forms a chain complex,
the rack complex of $X$.
Its homology groups $H^R_n(X)$ constitute the rack homology of $X$.
One purpose of this paper is to determine the rack homology of $R_p$ 
for $p$ an odd prime.
Homology and cohomology with coefficients in an abelian group $A$
are defined in the usual way.

Let $C^D_n(X)$ be the subgroup of $C^R_n(X)$ generated
by the $(x_1,\dots,x_n)$ such that $x_i=x_{i+1}$ for some $i$.
If $X$ is a quandle these constitute 
a subcomplex of the rack complex, 
called the degeneracy complex.
This is not true for  a general rack $X$.
The quotient groups $C^Q_n(X)=C^R_n(X)/C^D_n(X)$
form the quandle complex of $X$.
Its homology groups $H^Q_n(X)$ constitute the quandle homology of $X$.
A second purpose of this paper is to determine the quandle homology of $R_p$.
We will do this from the rack homology, using a theorem
of  \cite{lithern} which says that the canonical map from
rack homology to quandle homology splits.

\subsection{Known facts.}
It is noted in \cite{eisermann5} that calculating quandle cohomology is
difficult, since brute force calculations are very limited in range,
and unlike group cohomology, the topological underpinnings 
are less well developed.
It is our purpose to begin to remedy this situation
by showing how methods from homotopy theory can be applied.

The following list provides the  main facts which were already known,
and motivated our research.
\begin{itemize}
\item
In \cite{etingofg} a formula is proved for the 
dimension of $H^R_n(X;{\QQ})$ for $X$ a finite rack.
In particular for a connected quandle
these dimensions are all one, as they are for the one point rack.
This means that the interesting things happen in finite characteristic.
\item
In \cite{mochizuki2} the third cohomology is computed for Alexander
quandles associated to a finite field $k$ where $T$ is  multiplication by some $w\in k^*$.
Unfortunately the statement  of the main theorem and its proof contain some mistakes,
which have however been corrected in \cite{mande}.
\item
In \cite{lithern} it is proved that the torsion subgroup of  $H^R_n(X)$
is annihilated by $d^n$ if $X$ is a rack a cardinality $d$ with homogeneous orbits.
This is the case for Alexander racks.
In particular all torsion in the homology of $R_p$ is $p$-primary.
Thus it is sensible to concentrate first on the homology with coefficients in ${\ZZ}/(p)$.
\item
In \cite{niebp2} it is proved that for $p=3$ the torsion in the homology of $R_p$
is in fact of exponent $p$, and conjectured that this might be true for general $p$.
We will see that this is indeed the case.
\item
The same authors construct a homomorphism $h_a\colon H^Q_n(R_p)\to H^Q_{n+2}(R_p)$
and report on computer calculations showing that this map is a monomorphism
for small $n$ and $p$.
These calculations also suggests that the ranks of these groups 
form a `delayed Fibonacci sequence'.
We will generalize their construction and show that these conjectures are all true.
\end{itemize}

\subsection{A sketch of the new results.}
The explicit calculations in this paper deal with the rack cohomology of $R_p$ with 
coefficients in $\FF_p$.
It turns out that it differs only by a dimension shift
from the cohomology of a space $M=B(D_p;R_p)$
which is described in the next section,
and which carries a monoid structure $\mu\colon M\times M\to M$.
\begin{itemize}
\item
The cohomology vector spaces $H^n_R(M)$ have a basis consisting of expressions of the form
\begin{equation}
A_{m_1}\cup B^{e_1}\cup P^{j_1}(A_{m_2}\cup B^{e_2}\cup P^{j_2}(\dots  A_{m_s}\cup B^{e_s}\cup P^{j_s}(1))
\end{equation}
with $m_i,j_i\in\{0,1,2\dots\}$ and $e_i\in\{0,1\}$.
Here $A_m$ is in $H^{2m}_R (M)$, $B$ is in $ H^3_R(M)$,
and the additive operators $P\colon H^j_R(M)\to H^{j+1}_R(M)$ are Rota-Baxter operators.
The $A_m$ form a system of divided powers in the sense that
$A_mA_k=\binom{m+k}{k}A_{k+m}$.
\item
The homology is generated as an algebra by  generators $r\in H_1^R(M)$,
$s\in H_2^R(M)$ and $t\in H_3^R(M)$ with as only relations
$st=ts$ and $t^2=0$.
\end{itemize}
The action of the Bockstein operator on (co)homology is known
and shows that all torsion is of exponent $p$.
By the splitting result in \cite{lithern} the same is true for quandle (co)homology.
The above result allows us to estimate the quandle homology from above
and the quandle cohomology from below, and since both estimations coincide
the homology and cohomology are completely determined.
They can be expressed in a similar way as the rack (co)homology.
The main difference is that the operator $P$ is replaced by a similar
operator $Q$ satisfying $Q^2=0$, and that $r^2=0$.

\subsection{Organization.}
\label{notat}

The paper is organized as follows.
In $\S2$ we discuss the monoid structure on augmented rack space
and the resulting algebraic structure on chain level.
The formulas involving the cup product are proved in $\S3$.
In $\S4$ we apply A Serre spectral sequence to 
dihedral rack space and deduce the additive structure of its cohomology.
In $\S5$ we compute its algebra structure
and in $\S6$ its coalgebra structure.
In $\S7$ we show that all homology is of exponent $p$.
Finally we compute the quandle cohomology 
from the rack cohomology.

In order to simplify notation we will always use the same symbol for an 
operator acting on chains and the dual operator acting on cochains.
Thus for example in section \ref{xsc} we have operators 
$P$ and $D$ acting on chains such that $PD$ is the identity,
but on cochains $DP$ is the identity.

\section{The rack space and its coverings.}

\subsection{Groups associated to a rack.}
For a rack $X$ the \emph{adjoint group}  $Adj(X)$ is defined
as the group with a generator $e_a$ for each $a\in X$,
and relations $e_b^{-1}e_ae_b=e_{a*b}$ for $a,b\in X$.
There is a  canonical map $adj\colon X\to Adj(X)$
mapping $a\in X$ to $e_a\in Adj(X)$.
The functor $Adj$ is left adjoint to the functor $Conj$.
For each $b\in X$ the map $\sigma_b$ defined by 
$\sigma_b(a)=a*b$ is a rack automorphism of $X$.
For this reason racks were called `automorphic sets' in \cite{brieskorn}.
In this way we get a map
$\sigma\colon Adj(X)\to Aut(X)$.
The image is called the group $Inn(X)$ of \emph{inner} automorphisms of $X$.
If $Aut(X)$ acts transitively on $X$ then $X$ is called \emph{homogeneous}.
If even $Inn(X)$ acts transitively then $X$ is called \emph{connected}.
An Alexander quandle $Alex(M,T)$ is connected iff $1-T$ is invertible.

\subsection{Augmented racks}
An \emph{augmented} rack is a quadruple $(X,G,\eta,\rho)$
where $X$ is a rack, $G$ a group, $\rho$ a right action of $G$ on $X$ by rack homomorphisms,
and $\eta\colon X\to G$ a map which satisfies
\begin{equation}
\eta\rho(a,g)=g^{-1}\eta(a)g
\end{equation}
In this case the map $\eta$ extends uniquely to a homomorphism $\eta\colon Adj(X)\to G$
such that the composition with $\rho\colon G\to Aut(X)$ is just $\sigma$.
See \cite{joyce1} and \cite{fennrs1}.
For a  quandle we also demand that $\rho(a,\eta(a))=a$.\\
Examples:
\begin{itemize}
\item
For any rack $X$ take $G=Adj(X)$ and $\eta(a)=e_a$ and $\rho(x,e_a)=x*a$.
\item
An oriented manifold $M$ with an oriented properly embedded codimension $2$
submanifold $K$ and a point in $M-K$ defines an augmented quandle.
See  \cite{yetter3}.
\end{itemize}

In the above situation a right action of $G$ on a set $Y$ gives rise to a pairing
$Y\times X\to Y$ given by $y\star x=y\eta(x)$. 
It satisfies
\begin{enumerate}
\item For any $a\in Y$ and $b\in X$ there is a unique $c\in Y$ such that $a=c\star b$.
\item For any $a\in Y$ and $b,c\in X$ we have $(a\star b)\star c=(a\star c)\star(b*c)$.
\end{enumerate}
Such a pairing is called an \emph{action} of $X$ on $Y$,
and $Y$ is called an $X$-set.
Examples:
\begin{itemize}
\item
One can take $Y=X$, in which case $\star=*$.
\item
One can take  $Y=\infty$, the one point set.
\item
In \cite{fennrs1} an action of $X$ on $G$ is defined by 
$g\star x=\eta(x)^{-1}g$.
We prefer to take  $g\star x=g\eta(x)$,
coming from the action of $G$ on $G$ by right multiplicaton.
Obviously both actions are isomorphic by mapping $g$ to $g^{-1}$.
\end{itemize}

\subsection{The chain complex of an $X$-set. \label{xsc}}
For a rack $X$ acting on a set $Y$ we introduce a chain complex
as follows.
Let $C_n(Y;X)$ be the free abelian group generated by $Y\times X^n$.
Define a map $\partial\colon C_n(Y;X)\to C_{n-1}(Y;X)$ as follows:
\begin{equation}
\begin{split}
&\partial^0_i(y;x_1,\dots,x_n)=(y;x_1,\dots,x_{i-1},x_{i+1},\dots,x_n)\\
&\partial^1_i(y;x_1,\dots,x_n)=(y\star x_i;x_1*x_i,\dots,x_{i-1}*x_i,x_{i+1},\dots,x_n)\\
&\partial^0=\sum_{i=1}^n(-1)^i\partial^0_i,\qquad
\partial^1=\sum_{i=1}^n(-1)^i\partial^1_i,\qquad
\partial=\partial^0-\partial^1
\end{split}
\end{equation}
If $Y$ and $Z$ are $X$-sets then a map $f\colon Y\to Z$ 
is called a map of $X$-sets if $f(y\star x)=f(y)\star x$. 
Obviously every map of $X$-sets gives rise to a chain map.
In particular the unique map $Y\to\infty$ gives rise to a chain map
$\pi\colon C_n(Y;X)\to C_n(\infty;X)=C_n(X)$.

\begin{proposition}
\label{shift}
The isomorphisms
$\psi\colon C_n(X)\to C_{n-1}(X;X)$
given by 
\begin{equation}
\psi(x_1,\dots,x_n)=(-1)^{n-1}(x_1;x_2,\dots,x_n)
\end{equation}
for $n>0$ form a chain map.
\end{proposition}
\begin{proof}
Immediate from the fact that $\partial^1_1=\partial^0_1$.
\end{proof}

\begin{remark}
\label{defp}
Combining  $\psi$ with $\pi$ we get a chain map
$P=\psi\pi\colon C_{n-1}(X;X)\to C_{n-2}(X;X)$ described by
\begin{equation}
P(x_1;x_2,\dots,x_n)=(-1)^n(x_2;x_3,\dots,x_n)
\end{equation}
\end{remark}

\begin{proposition}
\label{sect}
Let $X$ be  a quandle.
Then the maps $D\colon C_{n-1}(X;X)\to C_n(X;X)$
given  by
\begin{equation}
D(x_1;x_2,\dots,x_n)=(-1)^{n-1}(x_1;x_1,x_2,\dots,x_n)
\end{equation}
form a chain map such that $PD$ is the identity.
This implies that the map $\pi_*\colon H_n(X;X)\to H_n(X)$ is surjective.
\end{proposition}
\begin{proof}
Straightforward.
\end{proof}

\subsection{The monoid stucture and homology operations.}
\begin{proposition}
Let $(X,G,\eta,\rho)$ be an augmented rack, 
and let $G$ act on $Y$.
Then there is a pairing
$\mu_Y\colon C_m(Y;X)\otimes C_k(G;X)\to C_{m+k}(Y;X)$
defined by
\begin{equation}
\mu_Y((y;x_1,\dots,x_m)\otimes(g,x'_1,\dots,x'_k))
=(yg;x_1g,\dots,x_mg,x'_1,\dots,x'_k)
\end{equation}
Here $xg$ is short for $\rho(x,g)$.
In particular there is a pairing 
\begin{equation}
\mu_G\colon C_k(G;X)\otimes C_\ell(G;X)\to C_{k+\ell}(G;X)
\end{equation}
Moreover one has
\begin{equation}
\mu_Y(\mu_Y(\eta\otimes\gamma)\otimes\gamma')=\mu_Y(\eta\otimes\mu_G(\gamma\otimes\gamma'))
\end{equation}
and
\begin{equation}
\partial\mu_Y(\eta\otimes\gamma)=\mu_Y(\partial\eta\otimes\gamma+(-1)^m\eta\otimes\partial\gamma)
\end{equation}
and if $f\colon Y\to Z$ is equivariant then 
\begin{equation}
\mu_Z(f(\eta)\otimes\gamma)=f\mu_Y(\eta\otimes\gamma)
\end{equation}
\end{proposition}
\begin{proof}
Straightforward from identities like $(yg)\star(xg)=(y\star x)g$, 
which follow from the definition of $\star$.
\end{proof}

From this proposition it follows that $\mu_G$ induces
a structure of associative algebra on the total homology of $(G;X)$.
Moreover the total homology of $(Y;X)$ is a right module over this algebra,
and any $G$-equivariant map $Y\to Z$ induces a module map.
For example  $\pi_*\colon H_n(X;X)\to H_n(X)$ is a module map.
In this paper we will determine this structure for the case that $X=R_p$.

\begin{remark}
\label{niebrem}
One can view the elements of $H_k(G;X)$ as additive operations of degree $k$
acting from the right on the homology of $X$.
Indeed the operations described in \cite{niebp2} can be viewed in this way:
\begin{itemize}
\item
The operation $h'_a$ is the one associated to the
class of $(e_a;a)+(1;a)$.
\item
The operation $h_s$ is the one associated to the class of
$\sum_{j=0}^{p-1}(1;j,j+1)$.
\end{itemize}
\end{remark}

\begin{remark}
\label{iota}
We will write $\one$ for the $0$-cochain given by $\one(y;)=1$ for every $y\in Y$.
Choose a base point $y\in Y$.
Then by applying $\mu_Y$ to $(y;)\in C_0(Y;X)$ we get a map
$\chi\colon C_k(G;X)\to C_k(Y;X)$.
If the action of $G$ on $Y$ is transitive then the class
of $(y;)$ in $H_0(Y;X)$ is independent of $y$ and will be denoted by  $\iota_Y$.
Thus $\chi\colon H_k(G;X)\to H_k(Y;X)$ is decribed by
$\chi(a)=\mu(\iota_Y\otimes a)$.
In particular the associativity of $\mu$ implies that
$\chi\mu_G=\mu_Y(\chi\otimes1)$.
\end{remark}

\begin{proposition}
\label{mup}
The interaction of $\mu_X$ with $P$ is given by
\begin{equation}
P\mu(a\otimes b)=(-1)^m \mu(P(a)\otimes b)+{\one}(a) P(\chi(b))
\end{equation}
for $a\in C_k(X;X)$ and $b\in C_m(G;X)$.
\end{proposition}

\begin{proof}
For $k>0$ one has
\begin{equation*}
\begin{split}
&P\mu_X\bigl((y;x_1,\dots,x_k)\otimes(g;x'_1,\dots,x'_m)\bigr)\\
&\quad=P(yg;x_1g,\dots,x_kg,x'_1,\dots,x'_m)\\
&\quad=(-1)^{k+m+1}(x_1g;x_2g,\dots,x_kg,x'_1,\dots,x'_m)\\
&\quad=(-1)^{k+m+1}\mu_X\bigl((x_1;x_2,\dots,x_k)\otimes(g;x'_1,\dots,x'_m)\bigr)\\
&\quad=(-1)^{k+m+1}(-1)^{k+1}\mu_X\bigl(P(y;x_1,\dots,x_k)\otimes(g;x'_1,\dots,x'_m)\bigr)\\
\end{split}
\end{equation*}
and for $k=0$
\begin{equation*}
\begin{split}
&P\mu_Y\bigl((y;)\otimes(g;x'_1,\dots,x'_m)\bigr)
=P(yg;x'_1,\dots,x'_m)\\
&\quad=1\cdot P(yg;x'_1,\dots,x'_m)
=\one(y;)\cdot P(\chi(g;x'_1,\dots,x'_m))\\
\end{split}
\end{equation*}
\end{proof}

\begin{proposition}
\label{mud}
The interaction of $\mu_X$ with $D$ is given by
\begin{equation}
D\mu_X(a\otimes b)=(-1)^m\mu_X(Da\otimes b)
\end{equation}
for $a\in C_k(X;X)$ and $b\in C_m(G;X)$.
\end{proposition}

\begin{proof}
We have
\begin{equation*}
\begin{split}
&D\mu_X\bigl(y;x_1,\dots,x_k)\otimes(g;x'_1,\dots,x'_m)\bigr)\\
&\quad=D(yg;x_1g,\dots,x_kg,x'_1,\dots,x'_m)\\
&\quad=(-1)^{k+m}(yg;xg,x_1g,\dots,x_kg,x'_1,\dots,x'_m)\\
&\quad=(-1)^{k+m}\mu_X\bigl((y;y,x_1,\dots,x_k)\otimes(g;x'_1,\dots,x'_m)\bigr)\\
&\quad=(-1)^{k+m}(-1)^k\mu_X\bigl(D(y;x_1,\dots,x_k)\otimes(g;x'_1,\dots,x'_m)\bigr)\\
&\quad=(-1)^m\mu_X(D\otimes1)\bigl((y;x_1,\dots,x_k)\otimes(g;x'_1,\dots,x'_m)\bigr)\\
\end{split}
\end{equation*}
\end{proof}

\begin{remark}
If we abbreviate $\mu(a\otimes b)$ to $a\cdot b$ 
then for $Y=G$ the above formulas read
\begin{equation}
\begin{split}
&P(a\cdot b)=(-1)^m P(a)\cdot b+\one(a) Pb\\
&D(a\cdot b)=(-1)^m D(a)\cdot b
\end{split}
\end{equation}
Thus $P$ is a  graded Fox derivation of the algebra $H_{\bullet}(G;X)$,
with respect to $\one$.
Also  $D$ acts on $C_m(X;X)$ as left multiplication by $(-1)^mD(\iota_X)$.
\end{remark}

\begin{remark}
We identify $C^k\otimes C^m$ with the dual of $C_k\otimes C_m$ 
using the pairing $\langle\dots\rangle$ given by
\begin{equation}
\langle F\otimes G,a\otimes b\rangle =(-1)^{km} F(a) G(b)
\end{equation}
Thus the cochain version of proposition  \ref{mup} reads
\begin{equation}
\mu(PF)=(P\otimes 1)\mu F+\one\otimes PF
\end{equation}
and the cochain version of proposition  \ref{mud} reads
\begin{equation}
\mu(DF)=(D\otimes 1)\mu F
\end{equation}
\end{remark}

\subsection{The rack space as a monoid.}
We now turn to the topological constructions which give rise
to some of these complexes and chain maps.
From a rack $X$ acting on a set $Y$ 
the \emph{action rack space} $B(Y;X)$ is defined as in  \cite{fennrs1} as follows.
One starts with $Y\times([0,1]\times X)^n$ and defines an
equivalence relation by
\begin{equation}
\begin{split}
&(y;t_1,x_1,t_2,x_2,\dots,0,x_j,\dots,t_n,x_n)\\
&\quad\sim(y;t_1,x_1,\dots,t_{j-1},x_{j-1},t_{j+1},x_{j+1},\dots,t_n,x_n)\\
&(y;t_1,x_1,t_2,x_2,\dots,1,x_j,\dots,t_n,x_n)\\
&\quad\sim(y\star x_j;t_1,x_1* x_j,\dots,t_{j-1},x_{j-1}*x_j,t_{j+1},x_{j+1},\dots,t_n,x_n)\\
\end{split}
\end{equation}
and defines $B(Y;X)$ as the quotient space.

It is easy to see that the chain complex $\{C_n(Y;X)\}$
introduced before is just the cellular complex of this space,
with one cell $[0,1]^n$ for each $(n+1)$-tuple $(y;x_1,\dots,x_n)$.
Moreover the pairing $\mu$ of chain complexes 
is induced by a pairing $\mu$ of spaces given by
\begin{equation}
\begin{split}
&\mu([y;t_1,x_1,\dots,t_m,x_m],[g;t'_1,x'_k,\dots,t'_k,x'_k])\\
&=[yg;t_1,x_1g,\dots,t_m,x_mg,t'_1,x'_1,\dots,t'_k,x'_k]\\
\end{split}
\end{equation}
In particular we get a strictly associative monoid structure on $B(G;X)$.
Note that this monoid contains the group $G$ as a submonoid,
so that by restriction we get a pairing $B(Y;X)\times G\to B(Y;X)$,
which is the \emph{edge action} described in \cite{fennrs1}.
In case $Y=G$  we get by restriction a pairing $G\times B(G;X)\to B(G;X)$
which is the \emph{vertex action} described in the same paper
(but note our different convention).
According to theorem 3.7 of  \cite{fennrs1} and the remarks preceding it we have:
\begin{proposition}
Let $(X,G,\eta,\rho)$ be an augmented rack, with $G$ acting on $Y$.
Then the projection $B(Y;X)\to B(X)$ is a covering.
In particular $B(G;X)\to B(X)$ is a principal $G$-bundle,
with the vertex action as covering transformations.
Moreover  $B(Y;X)$ can be identified with $Y\times_G B(G;X)$.
\end{proposition}

\begin{proposition}
Assume that the map $\eta\colon Adj(X)\to G$ is surjective.
Then $B(G;X)$ is connected.
\end{proposition}

\begin{proof}
The vertices associated to $g$ and $ge_x$ are connected by the edge
associated to $(g;x)$.
Therefore any two vertices are connected,
\end{proof}

\begin{remark}
\label{trivact}
A first consequence of this is that the covering transformations 
act trivially on homology.
A second consequence is that $B(G;X)$ is homotopy equivalent to 
a loop space, since that is true for any connected
associative topological monoid.
\end{remark}

\subsection{The universal property of  $B(G;X)$.}

The following fact is stressed in \cite{fennrs1}:
if a finite set $X$ is equipped with the trivial rack structure then
$B(X)$ is nothing but the James construction applied to 
the suspension of $X$.
One aspect of the James construction $J(Y)$ applied to a space $Y$ is
that it yields the free topological monoid on $Y$.
This means that there is a map from $Y$ to $J(Y)$
which is universal among maps from $Y$ to a topological monoid.
We can give a similar interpretation to $B(G;X)$ 
and thus view it as some kind of generalized James construction.

\begin{proposition}
Let  $(X,G,\eta,\rho)$ be an augmented rack,
and let  $\CALM$ be a topological monoid.
There is a bijection
between monoid maps $\Phi\colon B(G;X)\to\CALM$
and pairs consisting of a map $\phi\colon G\to\CALM$
and a map $f\colon[0,1]\times X\to \CALM$ such that
\begin{itemize}
\item
$\phi$ is a homomorphism and $f$ is continuous.
\item
$f(0,x)$ is the identity of $\CALM$.
\item
$f(1,x)=\phi(\eta(x))$.
\item
$\phi(g)^{-1}f(t,x)\phi(g)=f(t,\rho(x,g))$.
\end{itemize}
\end{proposition}

\begin{proof}
Straightforward: one writes $(g;t_1,x_1,\dots,t_n,x_n)\in B(G;X)$ as the product of
$(g;)$ and $(1;t_1,x_1),\dots,(1,t_n,x_n)$.
Moreover one writes $\phi(g)$ for $\Phi((g;))$ and $f(t,x)$ for $\Phi((1;t,x))$.
\end{proof}

If we take the canonical choice $G=Adj(X)$ then this reduces to
\begin{proposition}
Let $X$ be a rack and let $\CALM$ be a topological monoid.
There is a bijective correspondence between monoid maps $\Phi\colon B(Adj(X);X)\to\CALM$
and maps $f\colon[0,1]\times X\to\CALM$ such that
\begin{itemize}
\item
$f(0,x)$ is the unit element.
\item
$f(1,x)$ is invertible.
\item
$f(1,y)^{-1}f(t,x)f(1,y)=f(t,x*y)$.
\end{itemize}
\end{proposition}

\begin{proof}
Straightforward.
\end{proof}

It is yet a mystery how this relates to the most important aspect of the 
James construction $J(Y)$ applied to a space $Y$: the fact that it provides
a homotopy model for the loop space of the suspension of $Y$.

\subsection{Simplicity of  $B(X)$.}

The fact that $B(G;X)$ is a topological monoid gives a nice
alternative explanation for the following fact noted in \cite{fennrs1}:
the canonical action of the fundamental group of $B(X)$ on its higher
homotopy groups is trivial.
It can be viewed as a case of the following general fact:
\begin{proposition}
Let $\CALM$ be a simply connected topological monoid 
containing a discrete subgroup $G$, so that the canonical 
projection $\CALM\to\CALM/G$ is a covering map.
Then the canonical action of the fundamental group of $\CALM/G$ on its higher
homotopy groups is trivial.
\end{proposition}

\begin{proof}
Let $\gamma$ be an element of the fundamental group,
represented by a loop in $\CALM/G$,
which is lifted to a curve $c$ in $\CALM$ starting at the base point $m$.
Moreover let $\xi$ be an element of $\pi_n(\CALM)$, which is 
represented by a map $f\colon[0,1]^n\to\CALM$ mapping 
the boundary of $[0,1]^n$ to $m$.
In order to find $\gamma\xi$ one has to construct a map
$F\colon[0,1]^n\times [0,1]\to\CALM/G$ such that $F(x,0)=f(x)$
and $F(0,t)=c(t)$;
then $\gamma\xi$ is represented by the map $t\mapsto F(x,1)$.
However here we can simply take $F(x,t)=f(x)\cdot c(t)$ using the monoid structure.
\end{proof}

\subsection{The cup product.\label{cupprod}}
As we have seen rack cohomology is in fact the cohomology
of a space, the rack space.
Therefore the topological cup product gives rise to a ring structure
on cohomology.
We now describe a product on the cochain complex informally.
In the next section we will describe it more formally and
prove that the induced product on cohomology
coincides with the topological cup product.
We will do this as as special case of a more general theorem
about $\Box$-sets.

For $f\in C^k(B(Y;X))$ and $g\in C^m(B(Y;X))$ the product $f\cup g$
applied to  a sequence $(y;x_1,\dots,x_{k+m})$ is a sum of terms,
corresponding to subsets $B$ of $(x_1,\dots,x_{k+m})$ of cardinality $m$,
as follows:
\begin{itemize}
\item
The arguments of $f$ are $y$
and the elements of $B$ in ascending order.
\item
The first argument of $g$ is $y$
after it is acted upon by all elements of $B$.
The remaining arguments are the elements $x_u$
of the complementary subset $A$,
after they are acted upon by the elements $x_v$ of $B$ with $v>u$.
\item
Lastly every term is preceded by a sign depending of the parity of the permutation involved.
\end{itemize}
The example $k=m=2$ may illustrate this: 
\begin{equation}
\begin{split}
(f\cup g)&(y;x_1,x_2,x_3,x_4)\\
=&f(y;x_1,x_2)g((y\star x_1)\star x_2;x_3,x_4)\\
&-f(y;x_1,x_3)g((y\star x_1)\star x_3;x_2*x_3,x_4)\\
&+f(y;x_1,x_4)g((y\star x_1)\star x_4;x_2*x_4,x_3*x_4)\\
&+f(y;x_2,x_3)g((y\star x_2)\star x_3;(x_1*x_2)*x_3,x_4)\\
&-f(y;x_2,x_4)g((y\star x_2)\star x_4;(x_1*x_2)*x_4,x_3*x_4)\\
&+f(y;x_3,x_4)g((y\star x_3)\star x_4;(x_1*x_3)*x_4,(x_2*x_3)*x_4)\\
\end{split}
\end{equation}
Moreover this  product is strictly associative on the cochain level,
and has $\one$ as a unit.

\begin{proposition}
\label{rota}
The interaction of the cup-product with $\psi$ is given by
\begin{equation}
\psi F\cup\psi G=\psi(F\cup PG)+(-1)^{k+1}\psi(PF\cup G)
\end{equation}
for $F\in C^k(X;X)$ and $G\in C^m(X;X)$.
\end{proposition}

\begin{remark}
\label{rota2}
By applying $\pi$ to the above formula we find
\begin{equation}
PF\cup PG=P(F\cup PG)+(-1)^{k+1}P(PF\cup G)
\end{equation}
for $F\in C^k(X;X)$ and $G\in C^m(X;X)$.
This means that $P$ is a (graded) Rota-Baxter operator with respect to the cup product.
See \cite{rota} and \cite{guo} and \cite{guofard} for more on Rota-Baxter algebras.
\end{remark}

\begin{proposition}
\label{deriv}
The interaction of the cup-product with $D$ is given by
\begin{equation}
D(F\cup G)=DF\cup G+(-1)^k F\cup DG
\end{equation}
for $F\in C^k(X;X)$ and $G\in C^m(X;X)$.
\end{proposition}
Thus on cochains $P$ has the formal properties of integration
and $D$ has the formal properties of differentiation.
Moreover $DP$ is the identity map.
Since we need the formal definition of $\cup$ using the language of $\Box$-sets
the proofs of these two propositions are given in the next section.

\subsection{Some remarkable identities.}

Let us write $\Lambda$ for the element $P({\one})\in C^1(Y;X)$.
Thus $\Lambda(y;x)=1$ for all $y\in Y$ and $x\in X$.
\begin{proposition}
One has $\Lambda\cup\Lambda=0$, and
\begin{equation}
\begin{split}
&\partial^0F=-F\cup\Lambda\\
&\partial^1F=(-1)^n\Lambda\cup F\\
\end{split}
\end{equation}
for $F\in C^{n-1}(Y;X)$.
\end{proposition}

\begin{proof}
Straightforward.
\end{proof}

\begin{proposition}
Let $Q\colon C^n(X;X)\to C^{n+1}(X;X)$ be defined by \begin{equation}
Q(F)=P(F)+(-1)^{n+1}F\cup\Lambda
\end{equation}
Then $Q^2=0$.
\end{proposition}

\begin{proof}
We have
\begin{equation}
\begin{split}
Q^2(F)
&=P(PF+(-1)^{n+1}F\cup\Lambda)+(-1)^n(PF+(-1)^{n+1}F\cup\Lambda)\cup\Lambda\\
&=PPF+(-1)^{n+1}P(F\cup\Lambda)+(-1)^nPF\cup\Lambda-F\cup\Lambda\cup\Lambda\\
\end{split}
\end{equation}
But 
$PF\cup\Lambda
=PF\cup P\one
=P(F\cup P\one)+(-1)^{n+1}PPF$
and $\Lambda\cup\Lambda=0$.
\end{proof}

This fact seems less an accident if one observes that $Q$ 
is conjugated to the operator $\partial^0$:
\begin{proposition}
\label{qisdel}
$\psi Q=-\partial^0\psi$  on $C^n(X;X)$.
\end{proposition}
\begin{proof}
Straightforward.
\end{proof}

\begin{proposition}
If $\delta G=0$ then
\begin{equation}
QF\cup QG=Q(F\cup QG)+(-1)^{k+1}Q(QF\cup G)
\end{equation}
for $F\in C^k(X;X)$ and $G\in C^m(X;X)$.
\end{proposition}

\begin{proof}
By definition of $Q$ and proposition \ref{rota} we have
\begin{equation}
\begin{split}
QF\cup QG
&=P(F\cup PG)
+(-1)^{k+1}P(PF\cup G)\\
&+(-1)^{m+1}PF\cup G\cup\Lambda)
+(-1)^{k+1}F\cup\Lambda\cup PG\\
&+(-1)^{k+m}F\cup\Lambda\cup G\cup\Lambda\\
Q(F\cup QG)
&=P(F\cup PG)
+(-1)^{m+1}P(F\cup G\cup\Lambda)\\
&+(-1)^{k+m}F\cup PG\cup\Lambda
+(-1)^{k+1}F\cup G\cup\Lambda\cup\Lambda\\
(-1)^{k+1}Q(QF\cup G)
&=(-1)^{k+1}P(PF\cup G)
+P(F\cup\Lambda\cup G)\\
&+(-1)^{m+1}PF\cup G\cup\Lambda
+(-1)^{k+m} F\cup\Lambda\cup G\cup\Lambda\\
\end{split}
\end{equation}
If $\delta G=0$ then $-G\cup\Lambda=\partial^0G=\partial^1G=(-1)^{m+1}\Lambda\cup G$ and therefore
\begin{equation*}
(-1)^{k+1}F\cup\Lambda\cup PG=(-1)^{k+m}F\cup PG\cup\Lambda
\end{equation*}
Moreover in that case $\delta PG=P\delta G=0$ which implies in a similar way that $-PG\cup\Lambda=(-1)^m\Lambda\cup PG$ so
$(-1)^{m+1}P(G\cup G\cup\Lambda)+P(F\cup\Lambda\cup G)=0$.
\end{proof}

Thus $Q$ does not behave as a Rota-Baxter operator on the the cochain level
but it does so on the cohomology level.

\section{Cup products in $\Box$-sets.}

\subsection{Introduction to $\Box$-sets.}
Since the rack spaces $B(Y;X)$ are built up from cubes,
we have to study general spaces constructed from cubes.
This is formalized in the theory of $\Box$-sets,
see \cite{fennrs2}.
\begin{definition}
Maps $d^\epsilon_i\colon[0,1]^{n-1}\to[0,1]^n$
are defined for $n\geq 1$ by
\begin{equation*}
d^\epsilon_i(t_1,\dots,t_{n-1})=(t_1,\dots,t_{i-1},\epsilon,t_i,\dots,t_n)
\quad\text{for }\epsilon\in\{0,1\}\text{ and }1\leq i\leq n
\end{equation*}
They satisfy $d^\epsilon_i d^\eta_{j-1}=d^\eta_j d^\epsilon_i$ for $1\leq i<j\leq n$.
\end{definition}

\begin{definition}
A $\Box$-set $X$ is a sequence of sets $X_n$ for $n=0,1,2,\dots$
together with \emph{face maps} $\delta^\epsilon_i\colon X_n\to X_{n-1}$
for $\epsilon\in\{0,1\}$ and $1\leq i\leq n$ such that
$\delta^\eta_{j-1}\delta^\epsilon_i=\delta^\epsilon_i\delta^\eta_j$ for $1\leq i<j\leq n$.
Its chain group $C_n(X)$ is defined as the free abelian group generated by $X_n$,
and the boundary operator $\partial\colon C_n(X)\to C_{n-1}(X)$ is defined
as $\sum_i(-1)^i(\delta^0_i-\delta^1_i)$.
\end{definition}

\begin{remark}
Some caution is needed when dealing with $\Box$-sets.
Note that the singular cubes in a topological space do \emph{not}
yield the correct homology, but after dividing out
the degenerate ones they do.
\end{remark}

\begin{definition}
\label{AB}
We write $[n]$ for $\{i\in\ZZ\mid1\leq i\leq n\}$.
Let $A\subseteq[n]$ and $B=[n]-A$,
say $A=\{a_1,a_2,\dots,a_m\}$ and $a_1<a_2<\dots<a_m$,
and $B=\{b_1,\dots,b_k\}$ with $b_1<\dots<b_k$.
Then we write $\epsilon(A)$ for the sign of the permutation $\sigma_A$
that maps $(1,2,\dots,n)$ to $(b_1,\dots,b_k,a_1,\dots,a_m)$.
Moreover if $X$ is a $\Box$-set  
we write $\delta^\epsilon_A=\delta^\epsilon_{a_1}\circ\dots\circ\delta^\epsilon_{a_m}$ 
\end{definition}

\begin{definition}
\label{boxcup}
Let $X$ be a $\Box$-set, and let $f\in C^k(X)$ and $g\in C^m(X)$.
Then $f\cup g\in C^{k+m}(X)$ is defined by
\begin{equation}
(f\cup g)(x)=(-1)^{km}\sum_A \epsilon(A) \cdot f(\delta^0_A(x)) \cdot g(\delta^1_B(x))
\end{equation}
where $B=[n]-A$ and the sum is over all subsets $A$ of cardinality $m$.
\end{definition}

\begin{definition}
The \emph{realization} $\norm{X}$ of a $\Box$-set $X$
is defined as the quotient of the topological sum
$\coprod_n X_n\times[0,1]^n$ by the identifications
$(\delta^\epsilon_i(x),t)\sim(x,d^\epsilon_i(t))$.
\end{definition}

The aim of this section is to show that the above formula for the cupproduct 
for cochains on $X$ agrees with the topological cupproduct on $\norm{X}$. 
To do this we rewrite $\norm{X}$ as the realization of a simplicial set.
The idea is to triangulate the cubes into simplices,
and to use the known formula in the simplicial case.
The proof in this section is adapted from \cite{mande}.

\subsection{The triangulation.}

The triangulation which we will describe will not deliver us 
an honest  simplicial set\, but one lacking degeneracies.
\begin{definition}
A $\Delta$-set $Y$ is a sequence of sets $Y_n$ for $n=0,1,2,\dots$
together with \emph{face maps} $\delta_i\colon Y_n\to Y_{n-1}$
for $0\leq i\leq n$ such that
$\delta_{j-1}\delta_i=\delta_i\delta_j$ for $0\leq i<j\leq n$.
Its chain group $C_n(Y)$ is defined as the free abelian group generated by $Y_n$,
and the boundary operator $\partial\colon C_n(Y)\to C_{n-1}(Y)$ is defined
as $\sum_i(-1)^i\delta_i$.
\end{definition}

\begin{definition}
A $k$-partition of $[n]$ is a  sequence $S=(S_1;S_2;\dots;S_k)$ of nonempty subsets of $[n]$
which are pairwise disjoint and have $[n]$ as their union.
\end{definition}

\begin{definition}
For a $\Box$-set $X$ we define a $\Delta$-set $T(X)$.
The set of $k$-simplices $T(X)_k$ consists of the pairs
$(x;S)$ where $x\in X_n$ and $S$ is a $k$-partition of $[n]$.
The boundary maps are given by 
\begin{equation*}
\begin{split}
&\delta_0(x;S_1;\dots;S_k))
=(\delta^1_{S_1}x;\theta_{S_1}(S_2);\dots;\theta_{S_1}(S_k))\\
&\delta_i(x;S_1;\dots;S_k)
=(x;S_1;\dots; S_{i-1}; S_i\cup S_{i+1}; S_{i+2};\dots;S_k)
\quad\text{for }0<i<k\\
&\delta_k(x;S_1;\dots;S_k)
=(\delta^0_{S_k}(x);\theta_{S_k}(S_1);\dots;\theta_{S_k}(S_{k-1}))\\
\end{split}
\end{equation*}
Here $\theta_S$ denotes for $S\subseteq[n]$
the unique order preserving map from $[n]-S$ to $[n-\#(S)]$.
\end{definition}
\begin{remark}
To check the necessary relations one uses that
$\theta_{S\cup T}=\theta_{\theta_S(T)}\circ \theta_S$ if  $S,T\subset[n]$ are disjoint,
and a similar formula for the $\delta^\epsilon_S$.\\
As particular cases we have
\begin{equation}
\begin{split}
&\delta_0^k(x;S_1;\dots;S_n)=(\delta^1_B(x);\theta_B(S_{k+1});\dots;\theta_B(S_n))\\
&\delta_{k+1}\dots \delta_n(x;S_1;\dots;S_n)=(\delta^0_A(x);\theta_A(S_1);\dots;\theta_A(S_k))\\
\end{split}
\end{equation}
where $A=S_{k+1}\cup\dots\cup S_n$ and $B=S_1\cup\dots\cup S_k$,
and where $\theta_B\colon A\to[\#(A)]$ and $\theta_A\colon B\to[\#(B)]$ are the unique order preserving maps.
\end{remark}

The geometrical $k$-simplex associated to $(x;S)$ is the subset of $\{x\}\times[0,1]^n$
consisting of the $(x;t_1,\dots,t_n)$ with the property:
if $\alpha\leq\beta$ and $i\in S_\alpha$ and $j\in S_\beta$ then $t_i\leq t_j$.\\
Consider the special case $k=n$.
An $n$-partition $S$ of $[n]$ can be viewed as
a permutation $\sigma\in\CALS_n$
using the formula $S_i=\{\sigma(i)\}$.\\
The $n$-simplex now consists of the points
$(x;t_1,\dots,t_n)\in X\times[0,1]^n$ for which
$\alpha\leq\beta$ implies  $t_{\sigma(\alpha)}\leq t_{\sigma(\beta)}$.
We write $\sigma(x)$ for this simplex.
It is clear that these simplices cover $\norm{X}$ 
and intersect only in a common face.\\

\begin{definition}
For $x\in X_n$ define
$\tau(x)=\sum_{\sigma\in\CALS_n} \epsilon(\sigma)\sigma(x)$.
\end{definition}

\begin{proposition}
$\tau$ induces a chain map from chain complex of 
the $\Box$-set $X$ to the chain complex of the $\Delta$-set $T(X)$.
\end{proposition}

\begin{proof}
We prove that $\delta(\tau(x))=\tau(\delta(x))$.
Let $\sigma\in\CALS_n$ and $0<i<n$ .
Then one of $\sigma(x)$ and $(i\;i+1)\sigma(x)$
looks like $(x;\dots;a;b;\dots)$
and the other looks like  $(x;\dots;b;a;\dots)$.
Therefore the $\delta_i$ of these two terms cancel in $\delta(\tau(x))$.
We are left with
\begin{equation*}
\delta(\tau(x))=\delta_0(\tau(x))+(-1)^n\delta_n(\tau(x))
\end{equation*}
Let $T_i=\{\sigma\in\CALS_n\mid\sigma(1)=i\}$,
and for $\sigma\in T_i$ define 
$\rho(j)=\theta_i\sigma(j+1)$ for $1\leq j\leq n-1$.
Then
$\rho=(i\;i+1\;\dots\;n)^{-1}\sigma(1\;2\,\dots\;n)\in\CALS_{n-1}$,
so $\epsilon(\rho)=(-1)^{n-i}\epsilon(\sigma)(-1)^{n-1}=(-1)^{i-1}\epsilon(\sigma)$.
For $\sigma\in T_i$ we find
\begin{equation*}
\begin{split}
\delta_0\sigma(x)
&=\delta_0\left(x;i\mid\sigma(2)\mid\dots\mid\sigma(n)\right)\\
&=\left(\delta^1_i(x);\theta_i\sigma(2)\mid\dots\mid \theta_i\sigma(n)\right)
=\rho\delta^1_i(x)\qquad\text{and thus}\\
\delta_0 \tau(x)
&=\delta_0\sum_{\sigma\in\CALS_n}\epsilon(\sigma)\sigma(x)
=\sum_i\sum_{\sigma\in T_i}\epsilon(\sigma)\delta_0\sigma(x)\\
&=\sum_i \sum_{\rho\in\CALS_{n-1}}(-1)^{i-1}\epsilon(\rho)\rho\delta^1_i(x)
=(-1)^{i-1}\tau\delta^1_i(x)
\end{split}
\end{equation*}
Let $T'_i=\{\sigma\in\CALS_n\mid\sigma(n)=i\}$,
and for $\sigma\in T'_i$ define $\rho'(j)=\theta_i\sigma(j)$ for $1\leq j\leq n-1$.
Then 
$\rho'=(i\;i+1\;\dots\;n)^{-1}\sigma\in\CALS_{n-1}$,
so $\epsilon(\rho')=(-1)^{n-i}\epsilon(\sigma)$.
For $\sigma\in T'_i$ we find
\begin{equation*}
\begin{split}
\delta_n\sigma(x)
&=\delta_n\left(x;\sigma(1)\mid\dots\mid\sigma(n-1)\mid i\right)\\
&=\left(\delta^0_i(x);\theta_i\sigma(1)\mid\dots\mid \theta_i\sigma(n-1)\right)
=\rho'\delta^0_i(x)\qquad\text{and thus}\\
\delta_n \tau(x)
&=\delta_n\sum_{\sigma\in\CALS_n}\epsilon(\sigma)\sigma(x)
=\sum_i\sum_{\sigma\in T'_i}\epsilon(\sigma)\delta_n\sigma(x)\\
&=\sum_i \sum_{\rho'\in\CALS_{n-1}}(-1)^{n-i}\epsilon(\rho')\rho'\delta^0_i(x)
=(-1)^{n-i} \tau\delta^0_i(x)
\end{split}
\end{equation*}
Summing over $i$ we get $\delta(\tau(x))=\sum_i \tau((-1)^i\delta^0_i(x)-(-1)^i\delta^1_i(x))=\tau\delta(x)$.
\end{proof}

\begin{definition}
\label{deltacup}
Let $Y$ be a $\Delta$-set, and let $f\in C^k(Y)$ and $g\in C^m(Y)$.
Then $f\cup g\in C^{k+m}(Y)$ is defined by
\begin{equation}
(f\cup g)(y)=(-1)^{km}f(\delta_{k+1}\dots \delta_ny)\cdot g(\delta_0^k y)
\end{equation}
for $y\in Y_{k+m}$.
This is the Alexander-Whitney formula as in theorem 8.5 of \cite{maclane1}.
\end{definition}

We now prove that the products in definition \ref{boxcup}
and definition \ref{deltacup} correspond under $\tau$.
\begin{proposition}
The dual map satisfies $\tau(f\cup g)=\tau(f)\cup\tau(g)$.
\end{proposition}

\begin{proof}
For any $A\subseteq[n]$ let $B$ and $\sigma_A$ be as in definition \ref{AB}.
Then $\sigma\in \CALS_n$ can be written uniquely as 
$\sigma_A(\sigma_1\times\sigma_2)$ for some $A$ and some
$\sigma_1\in\CALS_k$ and $\sigma_2\in\CALS_m$.\\
From $\theta_A\sigma(i)=\sigma_1(i)$ and $\theta_B\sigma(k+i)=\sigma_2(i)$
it follows that
\begin{equation}
\begin{split}
\delta_{k+1}\dots\delta_n\sigma y=\sigma_1\delta^0_A y\quad\text{and}\quad
\delta_0^k\sigma y=\sigma_2\delta^1_B y\\
\end{split}
\end{equation}
Thus for $f\in C^k(T(X))$  and $g\in C^m(T(X))$ and $y\in X_{k+m}$ we have
\begin{equation}
\begin{split}
(f\cup g)(\sigma y)
&=(-1)^{km}f(\delta_{k+1}\dots\delta_n\sigma y)\cdot g(\delta_0^k\sigma y)\\
&=(-1)^{km}f(\sigma_1\delta^0_A y)\cdot g(\sigma_2\delta^1_B y)
\end{split}
\end{equation}
Therefore
\begin{equation}
\begin{split}
(\tau &(f\cup g))(y)=(f\cup g)(\tau y)
=(f\cup g)\sum_\sigma \epsilon(\sigma)\sigma(y)\\
&=(f\cup g)\sum_A\sum_{\sigma_1\in\CALS_k}\sum_{\sigma_2\in\CALS_m}
\epsilon(\sigma_A(\sigma_1\times\sigma_2))\cdot(\sigma_A(\sigma_1\times\sigma_2)y)\\
&=\sum_A\sum_{\sigma_1\in\CALS_k}\sum_{\sigma_2\in\CALS_m}
\epsilon(A)\cdot\epsilon(\sigma_1)\cdot\epsilon(\sigma_2)\cdot(f\cup g)(\sigma_A(\sigma_1\times\sigma_2)y)\\
&=(-1)^{km}\sum_A\sum_{\sigma_1\in\CALS_k}\sum_{\sigma_2\in\CALS_m}
\epsilon(A)\cdot \epsilon(\sigma_1)\cdot\epsilon(\sigma_2)\cdot f(\sigma_1\delta^0_A y)\cdot g(\sigma_2\delta^1_B y)\\
&=(-1)^{km}\sum_A\epsilon(A)
\sum_{\sigma_1\in\CALS_k}f(\epsilon(\sigma_1)\cdot\sigma_1\delta^0_A y)
\sum_{\sigma_2\in\CALS_m}g(\epsilon(\sigma_2)\cdot\sigma_2\delta^1_B y)\\
&=(-1)^{km}\sum_A \epsilon(A)\cdot f(\tau\delta^0_Ay)\cdot g(\tau\delta^1_B y)\\
&=(-1)^{km}\sum_A \epsilon(A)\cdot(\tau f)(\delta^0_Ay)\cdot (\tau g)(\delta^1_B y)
=(\tau f\cup\tau g)(y)
\end{split}
\end{equation}
\end{proof}

\begin{remark}
\label{assoc}
The cup product in definition \ref{boxcup}
is strictly associative since the 
Alexander-Whitney cup product is strictly associative
and the map $\tau$ is strictly homomorphic and surjective.
Moreover the $0$-cochain which maps every vertex to $1$
is a strict unit.
\end{remark}

\begin{proposition}
Let $X$ be a $\Box$-set.
Then there is a chain equivalence between  $C^\bullet(X)$ 
and the singular cochains on $\norm{X}$
under which the product in definition \ref{boxcup}
corresponds to the cup product on singular cochains.
\end{proposition}

\begin{proof}
There is a functor $G$ from $\Delta$-sets to simplicial sets
which is left adjoint to the forgetful functor $F$ from simplicial set to $\Delta$-sets.
For each $\Delta$-set $Y$ there is a chain equivalence (see \cite{maclane1} theorem 8.6.1) from
the chain complex of $GY$ to the normalized chain complex of $GY$,
which coincides with the chain complex of $Y$.
Under this equivalence the Alexander-Whitney maps agree.

Now $C^\bullet(X)$  with the product of definition \ref{boxcup}
is equivalent to $C^\bullet(TX)$ with the AW product,
which is equivalent to $C^\bullet(GTX)$ with the AW product,
which is equivalent to the singular cochains on $\norm{GTX}$ with the AW product.
But $\norm{GTX}$ is homeomorphic to $\norm{X}$.
\end{proof}

\begin{remark}
The geometric realization  of the $\Delta$-set $Y$
is homeomorphic to the geometric realization of the simplicial set $GY$.
However the realization of a simplicial set $Z$
is not homeomorphic to the realization of the $\Delta$-set $FZ$.
They are however homotopy-equivalent.
For more on this see \cite{rourkes4}.
\end{remark}

Now we specialize to rack spaces.
\begin{definition}
\label{byx}
Let $X$ be  rack and let $Y$ be an $X$-set.
Then we get a $\Box$-set $\BB(Y;X)$ by defining $\BB(Y;X)_n=Y\times X^n$
and
\begin{equation}
\begin{split}
&\delta^0_i(y;x_1,\dots,x_n)=(y;x_1,\dots,x_{i-1},x_{i+1},\dots,x_n)\\
&\delta^1_i(y;x_1,\dots,x_n)=(y\star x_i;x_1*x_i,\dots,x_{i-1}*x_i,x_{i+1},\dots,x_n)\\
\end{split}
\end{equation}
\end{definition}
Obviously $\norm{\BB(Y;X)}$ is just what we called $B(Y;X)$,
and $C_n(\BB(Y;X))$ is what we called $C_n(Y;X)$.
Combining this definition with definition \ref{boxcup}
one finds the prescription of subsection \ref{cupprod}.

\subsection{The proof of proposition \ref{rota}.}

Now we investigate the relation of this cup product
with $\psi$ and thus $P$ and with $D$.
First we have to translate $\psi$  in the language 
of $\Box$-sets.
We write $\psi_\bullet\colon\BB(X)_n\to\BB(X;X)_{n-1}$
for the map given by 
\begin{equation}
\psi_\bullet(x_1,x_2,\dots,x_n)=(x_1;x_2,\dots,x_n)\\
\end{equation}
The relation with $\psi$ is given by
\begin{equation}
(\psi F)(x)=F(\psi x)=(-1)^{n-1}F(\psi_\bullet x)
\end{equation}
for $x\in\BB(X)_n$ and $F\in C^{n-1}(X;X)$.

\begin{proposition}
\label{dpsi}
One has
\begin{equation}
\begin{split}
&F(\psi_\bullet\delta^\epsilon_i x)=F(\delta^\epsilon_{i-1}\psi_\bullet x)
\text{ for }i>1\text{ and }\\
&F(\psi_\bullet\delta^\epsilon_1 x)=(-1)^{n}(PF)(\psi_\bullet x)\\
\end{split}
\end{equation}
for $x\in\BB(X)_n$ and $F\in C^{n-2}(X;X)$.
\end{proposition}

\begin{proof}
For $\epsilon=1$ one has
\begin{equation}
\begin{split}
&F(\psi_\bullet\delta^1_i(x_1,\dots,x_n))
=F(\psi_\bullet(x_1*x_i,\dots,x_{i-1}*x_i,x_{i+1},\dots,x_n))\\
&\qquad=F(x_1*x_i;x_2*x_i,\dots,x_{i-1}*x_i,x_{i+1},\dots,x_n)\\
&\qquad=F(\delta^1_{i-1}(x_1;x_2,\dots,x_n))
=F(\delta^1_{i-1}\psi_\bullet(x_1,x_2,\dots,x_n))\\
\end{split}
\end{equation}
for $i>1$ and
\begin{equation}
\begin{split}
&F(\psi_\bullet\delta^1_1(x_1,\dots,x_n))
=F(\psi_\bullet(x_2,\dots,x_n))
=F(x_2;x_3,\dots,x_n)\\
&=(-1)^{n}(PF)(x_1;x_2,\dots,x_n)
=(-1)^{n}(PF)(\psi_\bullet(x_1,x_2,\dots,x_n))\\
\end{split}
\end{equation}
and similarly for $\epsilon=0$.
\end{proof}

Now we prove proposition \ref{rota}.
Let $F\in C^k(X;X)$ and $G\in C^m(X;X)$ and $x\in\BB(X)_{k+m+2}$ then
by definition of the cup product we have
\begin{equation}
\begin{split}
&(\psi F\cup\psi G)(x)
=(-1)^{(k+1)(m+1)}\sum_A\epsilon(A)
\cdot (\psi F)(\delta^0_Ax)
\cdot(\psi G)(\delta^1_Bx)\\
&\qquad=(-1)^{km+k+m+1}\sum_A\epsilon(A)
\cdot F(\psi\delta^0_Ax)
\cdot G(\psi\delta^1_Bx)\\
&\qquad=(-1)^{km+k+m+1}\sum_A\epsilon(A)
\cdot (-1)^kF(\psi_\bullet\delta^0_Ax)
\cdot (-1)^mG(\psi_\bullet\delta^1_Bx)\\
&\qquad=(-1)^{km+1}\sum_A\epsilon(A)
\cdot F(\psi_\bullet\delta^0_Ax)
\cdot G(\psi_\bullet\delta^1_Bx)\\
\end{split}
\end{equation}
where the sum is over subsets $A\subset[k+m+2]$ of cardinality $m+1$,
and $B$ is the complement of $A$.
For each term we write
\begin{equation*}
\begin{split}
&A=\{a_1,a_2,\dots,a_{m+1}\}
\text{ with }
a_1<a_2<\dots<a_{m+1}\\
&B=\{b_1,b_2,\dots,b_{k+1}\}
\text{ with }
b_1<b_2<\dots<b_{k+1}
\end{split}
\end{equation*}
There are two possibilities:
\begin{itemize}
\item
If $a_1=1$ we write
\begin{equation*}
\begin{split}
&U=\{a-1\;;\;a\in A,a>1\}=\{a_2-1,a_3-1,\dots,a_{m+1}-1\}\\
&V=\{b-1\;;\;b\in B\}=\{b_1-1,b_2-1,\dots,b_{k+1}-1\}\\
\end{split}
\end{equation*}
From proposition \ref{dpsi} we get by induction
\begin{equation}
\begin{split}
F(\psi_\bullet\delta^0_Ax)
&=F(\psi_\bullet\delta^0_1\delta^0_{a_2}\dots\delta^0_{a_m}\delta^0_{a_{m+1}}x)\\
&=(-1)^{k}(PF)(\psi_\bullet\delta^0_{a_2}\dots\delta^0_{a_m}\delta^0_{a_{m+1}}x)\\
&=(-1)^{k}(PF)(\delta^0_{a_2-1}\dots\delta^0_{a_{m+1}-1}\psi_\bullet x)\\
&=(-1)^{k}(PF)(\delta^0_U\psi_\bullet x)\\
G(\psi_\bullet\delta^1_Bx)
&=G(\psi_\bullet\delta^1_{b_1}\delta^1_{b_2}\dots\delta^1_{b_k}\delta^1_{b_{k+1}}x)\\
&=G(\delta^1_{b_1-1}\delta^1_{b_2-1}\dots\delta^1_{b_k}\delta^1_{b_{k+1}-1}\psi_\bullet x)\\
&=G(\delta^1_V\psi_\bullet x)\\
\end{split}
\end{equation}
Finally if we write $\kappa$ for the cycle $(1\;2\;3\dots k+m+2)$ then
$\sigma_A(\kappa\sigma_U\kappa^{-1})^{-1}$ is the cycle
$(1\;b_1\;b_2\dots b_{k+1})$,
so $\epsilon(A)=(-1)^{k+1}\epsilon(U)$.
Therefore these terms add up to
\begin{equation}
\begin{split}
&(-1)^{km+1}\sum_U (-1)^{k+1}\epsilon(U)\cdot(-1)^{k}(PF)(\delta^0_U\psi_\bullet x)\cdot G(\delta^1_V\psi_\bullet x)\\
&=(-1)^{km}\sum_U \epsilon(U)\cdot (PF)(\delta^0_U\psi_\bullet x)\cdot G(\delta^1_V\psi_\bullet x)\\
&=(-1)^{m}(PF\cup G)(\psi_\bullet x)
=(-1)^{k+1}(PF\cup G)(\psi x)\\
&=(-1)^{k+1}(\psi(PF\cup G))(x)\\
\end{split}
\end{equation}
\item
If $b_1=1$ we write
\begin{equation*}
\begin{split}
&U=\{a-1\;;\;a\in A\}=\{a_1-1,a_2-1,\dots,a_{m+1}-1\}\\
&V=\{b-1\;;\;b\in B,b>1\}=\{b_2-1,b_3-1,\dots,b_{k+1}-1\}\\
\end{split}
\end{equation*}
From proposition \ref{dpsi} we get by induction
\begin{equation}
\begin{split}
F(\psi_\bullet\delta^0_Ax)
&=F(\psi_\bullet\delta^0_{a_1}\delta^0_{a_2}\dots\delta^0_{a_m}\delta^0_{a_{m+1}}x)\\
&=F(\delta^0_{a_1-1}\delta^0_{a_2-1}\dots\delta^0_{a_{m+1}-1}\psi_\bullet x)\\
&=F(\delta^0_U\psi_\bullet x)\\
G(\psi_\bullet\delta^1_Bx)
&=G(\psi_\bullet\delta^1_1\delta^1_{b_2}\delta^1_{b_3}\dots\delta^1_{b_{k+1}} x)\\
&=(-1)^{m}(PG)(\psi_\bullet\delta^1_{b_2}\delta^1_{b_3}\dots\delta^1_{b_{k+1}} x)\\
&=(-1)^{m}(PG)(\delta^1_{b_2-1}\delta^1_{b_3-1}\dots\delta^1_{b_{k+1}-1}\psi_\bullet x)\\
&=(-1)^{m}(PG)(\delta^1_V\psi_\bullet x)
\end{split}
\end{equation}
This time $\sigma_A(\kappa\sigma_U\kappa^{-1})^{-1}$ is the identity
so $\epsilon(A)=\epsilon(U)$.
Therefore these terms add up to
\begin{equation}
\begin{split}
&(-1)^{km+1}\sum_U \epsilon(U)\cdot F(\delta^0_U\psi_\bullet x)\cdot (-1)^m (PG)(\delta^1_V\psi_\bullet x)\\
&=(-1)^{km+m+1}\sum_U \epsilon(U)\cdot F(\delta^0_U\psi_\bullet x)\cdot (PG)(\delta^1_V\psi_\bullet x)\\
&=(-1)^{k+m+1}(F\cup PG)(\psi_\bullet x)
=(F\cup PG)(\psi x)\\
&=(\psi(F\cup PG))(x)\\
\end{split}
\end{equation}
\end{itemize}
Thus $\psi F\cup\psi G=\psi(F\cup PG)+(-1)^{k+1}\psi(PF\cup G)$.

\subsection{The proof of proposition \ref{deriv}.}

Now we investigate the relation of the cup product with $D$.
First we have to translate $D$ in the language 
of $\Box$-sets.
We write  $D_\bullet\colon\BB(X;X)_{n-1}\to\BB(X;X)_n$
for the map given by 
\begin{equation}
D_\bullet(x_1;x_2,\dots,x_n)=(x_1;x_1,x_2,x_3,\dots,x_n)
\end{equation}
The relation with $D$ is given by
\begin{equation}
(DF)(x)=F(Dx)=(-1)^{n-1}F(D_\bullet x)
\end{equation}
for $x\in\BB(X;X)_{n-1}$
and $F\in C^n(X;X)$.

\begin{proposition}
\label{dd}
One has
\begin{equation}
\begin{split}
&F(\delta^\epsilon_iD_\bullet x)=F(D_\bullet\delta^\epsilon_{i-1}x)
\quad\text{ for }i>1,\text{ and}\\
&F(\delta^\epsilon_1D_\bullet x)=F(x)\
\end{split}
\end{equation}
for $x\in\BB(X;X)_n$ and $F\in C^n(X;X)$
\end{proposition}

\begin{proof}
For $\epsilon=1$ one has
\begin{equation}
\begin{split}
&F(\delta^1_iD_\bullet(x_0;x_1,\dots,x_n))
=F(\delta^1_i(x_0;x_0,x_1,\dots,x_n))\\
&=F(x_0*x_{i-1};x_0*x_{i-1},\dots,x_{i-2}*x_{i-1},x_i,\dots,x_n)\\
&=F(D_\bullet(x_0*x_{i-1};x_1*x_{i-1},\dots,x_{i-2}*x_{i-1},x_i,\dots,x_n))\\
&=F(D_\bullet\delta^1_{i-1}(x_0;x_1,\dots,x_{i-1},x_i,\dots,x_n))
\end{split}
\end{equation}
for $i>1$ and
\begin{equation}
\begin{split}
F(\delta^1_1D_\bullet(x_0;x_1,\dots,x_n))
=F(\delta^1_1(x_0;x_0,x_1,\dots,x_n))
=F(x_0;x_1,\dots,x_n)
\end{split}
\end{equation}
and similarly for $\epsilon=0$.
\end{proof}

Now we prove proposition \ref{deriv}.
Let $F\in C^k(X;X)$ and $G\in C^m(X;X)$ and $x\in\BB(X;X)_{k+m-1}$ then
by definition of the cup product we have
\begin{equation}
\begin{split}
(D(F\cup G))(x)
&=(F\cup G)(Dx)
=(-1)^{k+m-1}(F\cup G)(D_\bullet x)\\
&=(-1)^{km+k+m+1}\sum_A \epsilon(A)\cdot F(\delta^0_AD_\bullet x)\cdot G(\delta^1_BD_\bullet x)
\end{split}
\end{equation}
where the sum is over subsets $A\subset[k+m]$ of cardinality $m$,
and $B$ is the complement of $A$.
For each term we write
\begin{equation*}
\begin{split}
&A=\{a_1,a_2,\dots,a_{m}\}
\text{ with }
a_1<a_2<\dots<a_{m}\\
&B=\{b_1,b_2,\dots,b_{k}\}
\text{ with }
b_1<b_2<\dots<b_{k}
\end{split}
\end{equation*}
There are two possibilities
\begin{itemize}
\item
If $a_1=1$ we write
\begin{equation*}
\begin{split}
&U=\{a-1\;;\;a\in A,a>1\}=\{a_2-1,a_3-1,\dots,a_{m}-1\}\\
&V=\{b-1\;;\;b\in B\}=\{b_1-1,b_2-1,\dots,b_{k}-1\}\\
\end{split}
\end{equation*}
From proposition \ref{dd} we get 
\begin{equation}
\begin{split}
F(\delta^0_AD_\bullet x)
&=F(\delta^0_1\delta^0_{a_2}\dots\delta^0_{a_m}D_\bullet x)
=F(\delta^0_1D_\bullet\delta^0_{a_2-1}\dots\delta^0_{a_m-1}x)\\
&=F(\delta^0_{a_2-1}\dots\delta^0_{a_m-1}x)
=F(\delta^0_Ux)\\
G(\delta^1_BD_\bullet x)
&=G(\delta^1_{b_1}\dots\delta^1_{b_k}D_\bullet x)
=G(D_\bullet\delta^1_{b_1-1}\dots\delta^1_{b_k-1}x)\\
&=G(D_\bullet\delta^1_Vx)
=(-1)^{m-1}(DG)(\delta^1_V x)
\end{split}
\end{equation}
If we write $\kappa$ for the cycle $(1\;2\;3\dots k+m)$ then
$\sigma_A(\kappa\sigma_U\kappa^{-1})^{-1}$ is the cycle
$(1\;b_1\;b_2\dots b_{k})$,
so $\epsilon(A)=(-1)^{k}\epsilon(U)$.
Therefore these terms add up to
\begin{equation}
\begin{split}
&(-1)^{km+k+m+1}\sum_U (-1)^{k}\epsilon(U)\cdot F(\delta^0_U x)\cdot (-1)^{m-1} (DG)(\delta^1_V x)\\
&=(-1)^{km}\sum_U \epsilon(U)\cdot F(\delta^0_U x)\cdot (DG)(\delta^1_V x)
=(-1)^k(F\cup DG)(x)\\
\end{split}
\end{equation}
\item
If $b_1=1$ we write
\begin{equation*}
\begin{split}
&U=\{a-1\;;\;a\in A\}=\{a_1-1,a_2-1,\dots,a_{m}-1\}\\
&V=\{b-1\;;\;b\in B,b>1\}=\{b_2-1,b_3-1,\dots,b_{k}-1\}\\
\end{split}
\end{equation*}
This time we have
\begin{equation}
\begin{split}
F(\delta^0_AD_\bullet x)
&=F(\delta^0_{a_1}\delta^0_{a_2}\dots\delta^0_{a_m}D_\bullet x)\\
&=F(D_\bullet\delta^0_{a_1-1}\dots\delta^0_{a_m-1}x)
=F(D_\bullet\delta^0_Ux)\\
&=(-1)^{k-1}(DF)(\delta^0_Ux)\\
G(\delta^1_BD_\bullet x)
&=G(\delta^1_{b_1}\delta^1_{b_2}\dots\delta^1_{b_k}D_\bullet x)\\
&=G(\delta^1_1D_\bullet\delta^1_{b_2-1}\dots\delta^1_{b_k-1}x)\\
&=G(\delta_{b_2-1}\dots\delta^1_{b_k-1}x)
=G(\delta^1_Vx)\\
\end{split}
\end{equation}
This time $\sigma_A(\kappa\sigma_U\kappa^{-1})^{-1}$ is the identity
so $\epsilon(A)=\epsilon(U)$.
Therefore these terms add up to
\begin{equation}
\begin{split}
&(-1)^{km+k+m+1}\sum_U \epsilon(U)\cdot(-1)^{k-1} (DF)(\delta^0_U x)\cdot  G(\delta^1_V x)\\
&=(-1)^{km+m}\sum_U \epsilon(U)\cdot (DF)(\delta^0_U x)\cdot G(\delta^1_V x)\\
&=(DF\cup G)(x)\\
\end{split}
\end{equation}
\end{itemize}
Thus $D(F\cup G)=DF\cup G+(-1)^k F\cup DG$.

\section{The key fibrations.}

\subsection{The cohomology of the coverings of $B(X)$.}

From now on we assume that $X$ is finite quandle with the following properties:
\begin{itemize}
\item
It is \emph{faithful} in the sense that
$x*a=x*b$ for all $x$ implies $a=b$.
In this case the the canonical map 
$X\to G=Inn(X)$ is injective.
We will identify $X$ with its image in $G$.
\item
It is connected: the action of $G$ on $X$ is transitive.
Thus there is a bijection $G/H\to X$
where $H$ is the isotropy group of some $a\in X$.
\item
It has `homogeneous orbits' so that
the result of \cite{lithern} can be applied
which says that the torsion in $H_n(X)$ 
is annihilated by $d^n$,
where $d$ is the cardinality of $X$.
\item
It is \emph{regular} in the sense that
the cardinalities of $X$ and $H$ are relatively prime.
\end{itemize}
The first three conditions are satisfied
for the Alexander quandle associated to $(M,T)$
if $1-T$ is invertible.
The last condition is satisfied if the order of $T$
is prime to the order of $M$.

All this is satisfied if $X$ is a \emph{Galois quandle},
where $M$ is a finite field $K$ of characteristic $p$,
and $T$ is multiplication by some $w\in K-\{0,1\}$.
Note that $G$ is a subgroup of the affine group of $K$.

Let $p$ be a prime dividing $d$, the cardinality of $X$.
As noted above only such a prime can be involved in the torsion
in the homology of $X$.
For this reason  we will start with looking a  the cohomology of $B(G,X)$ and $B(X,X)$ and $B(X)$
with coefficients in $\FF$, the field of $p$ elements.

A key role in our considerations is played by the following 
well known observation.
\begin{proposition}
\label{invar}
Let $\pi\colon Y\to X$ be a principal covering,
with group $\Gamma$.
If the order $d$ of $\Gamma$ is prime to $p$ then the map
\begin{equation}
\pi\colon H^n(X;\FF)\to H^n(Y,\FF)^\Gamma
\end{equation}
is an isomorphism.
\end{proposition}
\begin{proof}
The transfer map provides an inverse.
\end{proof}

For any augmented quandle $(X,G)$ one gets an equivariant  map
$G\to X$ by choosing some base point $x_0\in X$
and mapping $G$ to $x_0g$.
From this one gets a principal covering $B(G,X)\to B(X,X)$
with group the isotropy group of $x_0$.

In the situation considered here the zero element of $K$
is an obvious choice for $x_0$, and the group
$\Gamma$ consists of the powers of $w$.
So we get as a corollary:
\begin{proposition}
\label{bxxisbgx}
The projection map induces an isomorphism
\begin{equation}
\chi\colon H^n(B(X,X);\FF)\simeq H^n(B(G,X);\FF)
\end{equation}
\end{proposition}

\begin{proof}
By remark \ref{trivact} the action of $\Gamma$ on
the cohomology of $B(G,X)$ is trivial.
\end{proof}
Henceforward we will identify both cohomologies using $\chi$.
In particular proposition \ref{mup} now says that $\mu PF=(P\otimes1)\mu F+\one\otimes PF$.
Note also that  the element $D(\iota_X)$ mentioned in proposition \ref{mud}
corresponds under $\chi$ up to a factor $2$ with the class of $(1;y)+(e_y;y)$ 
which corresponds to the operation  $h'_a$
of \cite{niebp2}, as discussed in remark \ref{niebrem}.

Next we cite the result on page 349 of \cite{fennrs1},
again for general augmented racks:
\begin{proposition}
There is a map $\gamma$ from $B(X)$ to the classifying space $B(G)$
and the principal covering $B(G,X)\to B(X)$ is the pull-back
of universal covering $E(G)\to B(G)$.
\end{proposition}
Thus we have a commutative diagram
\begin{equation}
\xymatrix{
B(G,X)\ar[r]\ar[d]&E(G)\ar[d]\\
B(X)\ar[r]&B(G)\\
}
\end{equation}
In our special case this has the following consequence:
\begin{proposition}
$\gamma\colon H^n(B(G);\FF)\to H^n(B(X);\FF)$ vanishes for $n>0$.
\end{proposition}
\begin{proof}
It follows from proposition \ref{sect} and proposition \ref{bxxisbgx} that
the map $B(X)\to B(G;X)$ is injective in cohomology.
On the other hand the cohomology of $E(G)$ vanishes 
in positive dimensions.
\end{proof}

Like any map $\gamma$ can be replaced by an equivalent Hurewicz fibration.
One gets the fibre $F(\gamma)$ of this fibration by pulling back  the path space over $B(G)$.
Since the path fibration is equivalent to the covering $E(G)$ of $B(G)$,
the resulting fibre is equivalent to the pull back $B(G,X)$ of $E(G)$.
Henceforward we will make no difference in notation between any map
and the equivalent fibration that replaces it.

\begin{remark}
At this point one can see how a recursive computation of the cohomology
of $B(X)$ starting from the cohomology of $B(G)$ might be feasible.
One considers the fibration sequence
\begin{equation}
\label{fib1}
\xymatrix{
B(G;X)\ar[r]&B(X)\ar[r]^\gamma&B(G)
}
\end{equation}
where $\gamma$ is homologically trivial.
If one knows the cohomology of $B(X)$ up to dimension $n$
one can hope to be able to compute the cohomology of $B(G;X)$ up to dimension $n$
by a spectral sequence argument.
But this coincides with the cohomology of $B(X,X)$.
By proposition \ref{shift} this yields the cohomology of $B(X)$ up to dimension $n+1$.

In this way a new proof might be given of the results of \cite{mochizuki2} about $H^3$.
However from subsection \ref{dihed} onward we specialize to the dihedral case $K=\FF$, $w=-1$.
We try to recognize the pattern that emerges in $H^n$ for larger $n$,
and prove that the found pattern is the correct one
by using a spectral sequence comparison argument.
\end{remark}

\subsection{Replacing $B(G)$.}

A problem with the fibration $B(X;G)\to B(X)\to B(G)$
is the fact that the base space is not simply connected.
We will remedy this by replacing it by another space $L$
with the same cohomology which is simply connected.
\begin{definition}
Let $C\subset G$ be the cyclic group generated by $T$.
Let $i_C\colon B(C)\to B(G)$ the map of classifying spaces
induced by the inclusion $C\subset G$.
Then $L$ is defined to be the mapping cone of $i_C$.
We wil write $j$ for the inclusion $B(G)\to L$.
\end{definition}
So $L$ is built by attaching the cone of $B(C)$ to $B(G)$.
We might  reach our goal also  by just attaching one $2$-cell and one $3$-cell.

\begin{proposition}
The space $L$ is simply connected, and the map $j$ induces isomorphisms $H^n(B(G),\FF)\to H^n(L;\FF)$.
\end{proposition}

\begin{proof}
By the van Kampen theorem the effect of attaching a cone 
is quotienting out the normal subgroup generated by the image of the attaching map.
In the present case the normal subgroup generated by $T$ is the whole of $G$.
The second statement follows since $H^n(B(C);\FF)$ is trivial for $n>0$,
because the characteristic  of $\FF$ is prime to the order of $C$.
\end{proof}

Now we consider the following map of fibrations
\begin{equation}
\xymatrix{
F(\gamma)\ar[r]\ar[d]^J&B(X)\ar[r]^\gamma\ar[d]^1&B(G)\ar[d]^j\\
F(j\gamma)\ar[r]^\xi&B(X)\ar[r]^{j\gamma}&L\\
}
\end{equation}

\begin{proposition}
The map $J\colon H^n(F(j\gamma);\FF)\to H^n(F(\gamma);\FF)$ is an isomorphism.
\end{proposition}

\begin{proof}
We cite the Zeeman spectral sequence comparison theorem,
a version of which can be found as proposition 1.12 from \cite{hatcher2}:
Suppose we have a map of fibrations, 
and both fibrations satisfy the hypothesis of trivial action for the Serre spectral sequence. 
Then if two of the three maps induce isomorphisms on $R$-homology
with $R$ a principal ideal domain, so does the third.\\
In our case the fundamental group $G$ of $B(G)$ acts trivially on the cohomology of $F(\gamma)$
since it acts trivially on the cohomology of the equivalent space $B(G;X)$.
The fundamental group of $L$ acts trivially on the cohomology of $F(j\gamma)$
since it is trivial.
\end{proof}

\subsection{The second key fibration.}

The problem of computing the cohomology of $B(X;X)$ is now reduced to that
of computing the cohomology of $F(j\gamma)$.
To do that we change $\xi$ into a fibration $\xi'$.
This yields a fibration sequence
\begin{equation}
\label{fib2}
\xymatrix{
F(\xi')\ar[r]^\omega& F(j\gamma)\ar[r]^{\xi'}&B(X)
}
\end{equation}

\begin{proposition}
Let be given a fibration sequence
\begin{equation}
\xymatrix{F\ar[r]^j&E\ar[r]^p&B\\}
\end{equation}
and change $j$ into a fibration $j'\colon F'\to B$, giving a fibration sequence 
\begin{equation}
\xymatrix{G\ar[r]^k&F'\ar[r]^{j'}&E\\}
\end{equation}
then $G$ is homotopy equivalent to the loop space $\Omega B$,
and the action of $\pi_1(E)$ on $H_n(G)$ 
factorizes over the action of $\pi_1(B)$ on $H_n(\Omega B)$.
\end{proposition}

\begin{proof}
In general the action of $\pi_1(E)$ on $H_n(G)$ is induced by a pairing
$\Omega(E)\times G\to G$ which in turns is the restriction of a Hurewicz connection.
If $j'$ is coming from a fibration as indicated then such a Hurewicz connection
can be explicitly constructed from $p$, and the resulting pairing can be seen to factorize over $\Omega(B)$.
\end{proof}

Since the fundamental group of $L$ is trivial this means that in
the fibration sequence (\ref{fib2}) the fundamental group of the base 
acts trivially on the cohomology of the fibre, so that we
can set up a Serre spectral sequence.
The cohomology of the fibre is the cohomology of $\Omega(L)$
which we regard as known.
Now consider the following well known theorem.

\begin{proposition}
\label{leray}
(Leray-Hirsch).
Let be given a fibration sequence
\begin{equation}
\xymatrix{G\ar[r]^k&F'\ar[r]^{j'}&E\\}
\end{equation}
with $\pi_1(E)$ acting trivially on the $H^n(G,\FF)$.
Suppose that we can find elements $x_i\in H^{n_i}(F',\FF)$
such that the $k^*(x_i)$ form an $\FF$-basis of $H^\bullet(G,\FF)$.
Then the elements $x_i$ form a basis of $H^\bullet(F';\FF)$
as a module over the cohomology algebra of $H^\bullet(E;\FF)$, using $j'$.
\end{proposition}

\begin{remark}
\label{fibo}
Suppose that  this theorem is applicable to the fibration sequence  \ref{fib2}
then we have
\begin{equation}
\begin{split}
\dim H^{n+1}(B(X);\FF)
&=\dim H^n(B(X;X);\FF)=\dim H^n(F',\FF)\\
&=\sum_k \dim H^k(G;\FF)\cdot\dim H^{n-k}(B(X);\FF)\\
&=\sum_k \dim H^k(\Omega(L);\FF)\cdot\dim H^{n-k}(B(X);\FF)\\
\end{split}
\end{equation}
Thus we can compute the betti numbers of $B(X)$
from the known betti numbers of $\Omega(L)$.
Even better:
a basis of the cohomology of $B(X;X)$ is given by the expressions
\begin{equation}
x_{k_1}\cup P(x_{k_2}\cup P(x_{k_3}\dots
\end{equation}
where $P\colon H^n(B(X;X);\FF)\to H^{n+1}(B(X;X);\FF)$ is $\pi^*\psi$ as in remark \ref{defp}.\\
In any case we find a recursion formula for the betti numbers 
which for $X=R_p$ is a version of the recursion formula conjectured in
\cite{niebp2}.
The remainder of this section is devoted to the proof
that we are indeed in situation of the theorem, at least for
the case of the dihedral quandle $R_p$.
\end{remark}

\begin{remark}
From now on we shorten notation by writing the sequences \ref{fib1} and \ref{fib2} as
\begin{equation}
\xymatrix{\Omega L\ar[r]^\omega&M\ar[r]^\pi&BX
}
\quad\text{and}\quad
\xymatrix{M\ar[r]^\pi&BX\ar[r]^\gamma&L
}
\end{equation}
where $M$ stands for $B(G;X)$ or $B(X;X)$.
So these are fibration sequences up to homology equivalence.
\end{remark}

\subsection{Reverse transgression.}

Let be given a fibration sequence
\begin{equation}
\xymatrix{F\ar[r]^j&E\ar[r]^p&B\\}
\end{equation}
with $F$ the fibre over $b_0\in B$.
Suppose that $p^*$ is trivial in positive dimension.
Then consider the following diagram
\begin{equation}
\xymatrix{
&H^n(B;b_0;\FF)\ar[r]^\cong_{g^*}\ar[d]^{p^*}&H^n(B;\FF)\ar[d]^{p^*}\\
H^{n-1}(F;\FF)\ar[r]^\delta&H^n(E,F;\FF)\ar[r]&H^n(E;\FF)\\
}
\end{equation}
The diagram shows that for an element $\xi\in H^n(B;\FF)$
there is an element $\Xi\in H^{n-1}(F;\FF)$  
such that $\delta\Xi=p^*(g^*)^{-1}\xi$.
The fact that $\delta\Xi$ is in the image of $p^*$
shows that $\Xi$ is \emph{transgressive},
and the \emph{transgression} maps $\Xi$ to $\xi$ modulo indeterminacy.
It is well known (see page 54 of \cite{hatcher2}) that the transgression coincides
with the edge homomorphism $d_n\colon E_n^{0,n-1}\to E_n^{n,0}$
in the Serre spectral sequence of the fibration.
There are two situations in which this observation is relevant.
The first case is a path fibration, where $E$ is contractible.
The second case is the one where $p$ is $\gamma\colon B(X)\to L$.

\subsection{The cohomology of $L$ and $\Omega L$ in the dihedral case.\label{dihed}}

In this subsection we consider the dihedral quandle $R_p$.
In this case the group $G=Inn(R_p)$ is the dihedral group $D_p$.
\begin{proposition}
The cohomology of $B(D_p)$ is generated by an element $\alpha$
of degree $3$ and an element $\beta$ of degree $4$,
with $\alpha\cup\alpha=0$ as the only relation.
\end{proposition}

\begin{proof}
The cohomology $H^\bullet(B(C_p);\FF)$ is better known:
it has a generator $\theta$ of degree $1$
and a generator $\eta=\Delta(\theta)$ of degree $2$.
Here $\Delta$ denotes the Bockstein operator.
The covering $B(C_p)\to B(D_p)$ has one nontrivial
covering transformation which corresponds to inversion in $C_p$.
So it maps $\theta$ to $-\theta$ and therefore  $\eta$ to $-\eta$.
Now we use  proposition \ref{invar} for this double covering.
The algebra of invariants  under the covering transformation 
are generated by $\alpha=\theta\eta$ and $\beta=\eta^2$.
\end{proof}

\begin{proposition}
\label{cohol}
The cohomology of $\Omega L$ has a basis consisting
elements $A_k B^e\in H^{2k+3e}(\Omega L;\FF)$, 
where $k\geq0$ and $e\in\{0,1\}$.
Here $B^2=0$, and the $A_k$ constitute a  system of divided powers:
$A_k A_m=\binom{k+m}{k} A_{k+m}$.
\end{proposition}

\begin{remark}
We can construct $A_1$ by the reverse transgression argument from $\alpha$,
and $B$ similarly from $\beta$.
However we will not use this argument since we then still need to construct $A_p$.
\end{remark}

\begin{proof}
We use the Serre spectral sequence for the path fibration of $L$,
with coefficients in $\FF$.
Thus 
\begin{equation}
\begin{split}
E_2^{s,t}&=H^s(L;H^t(\Omega L;\FF))\cong H^s(L;\FF)\otimes H^t(\Omega L;\FF)\\
E_\infty^{s,t}~&=0\text{ if }(s,t)\not=(0,0)\\
\end{split}
\end{equation}
In particular we identify $E_2^{s,0}$ with $H^s(L;\FF)$ and $E_2^{0,t}$ with $H^t(\Omega L;\FF)$.
We claim that this spectral sequence has the following structure:
\begin{itemize}
\item
The cohomology of $\Omega L$ is as stated.
\item
$d_2A_k=0$ and $d_2B=0$.
\item
$d_3A_k=A_{k-1}\alpha$
and $d_3B=0$.
\item
$d_4B=\beta$.
\item
$d_r=0$ for $r>4$.
\end{itemize}
In other words 
\begin{itemize}
\item
$E_2$ and $E_3$ have a basis of monomials
$\alpha^f\beta^m A_kB^e\in E_2^{4m+3f,2k+3e}$ 
with $k,m\geq 0$ and $e,f\in\{0,1\}$.
\item
$E_4$ has a basis of monomials
$\beta^mB^e\in E_4^{4m,3e}$.
\item
$E_r$ has basis $1\in E_r^{0,0}$ for $r\geq 5$. 
\end{itemize}
We use induction.
The induction hypothesis $H(n)$ says 
that $E_r^{s,t}$ is as stated for $t+r-2\leq n$.
\begin{itemize}
\item
$n=1$.\\
Suppose that  $E_2^{0,1}$ contained an element $\xi$.
Then it would survive to $E_3$ since $d_2\xi\in E_2^{2,0}=0$.
And it would survive to $E_\infty$ since the higher $d_r$ 
point to some $E_r^{s,t}$ with $t<0$.
Thus $H^1(\Omega L)=0$.
\item
$n=2$.\\
The element $\alpha\in E_2^{3,0}$ is not hit by $d_2$ because $E_2^{1,1}=0$ since $H^1(L)=0$.
Suppose that it is not hit by $d_3$, and survives to $E^4$.
Then it survives to $E_\infty$ since the higher $d_r$ originate 
from some $E_r^{s,t}$ with $s<0$.
Therefore there must be some $A_1\in E_3^{0,2}\subset E_2^{0,2}=H^2(\Omega L)$ such that $d_3A_1=\alpha$.
If $E_2^{0,2}$ contained some $\xi$ independent from $A_1$ 
then $\xi$ would survive to $E_\infty$.
\item
$n=3$.\\
The element $\beta\in E_2^{4,0}$ is not hit by $d_3$ or $d_4$ since $H^1(L)=0=H^2(L)$.
If it were not hit by $d_4$ it would survive to $E_\infty$.
Therefore there must be some $B\in E_4^{0,3}\subset E_2^{0,3}=H^3(\Omega L)$ such that $d_4B=\beta$.
If $E_2^{0,3}$ contained anything more it would survive to $E_\infty$.
\item
$n=2k\geq 4$.\\
Consider the element $\alpha A_{k-1}\in E_2^{3,2k-2}$.
\begin{itemize}
\item
$d_2(\alpha A_{k-1})=0$ since $d_2\alpha=0$ and $d_2(A_{k-1})=0$.
\item
$d_3(\alpha A_{k-1})= \alpha^2 A_{k-2}=0$.
\item
$d_4(\alpha A_{k-1})$ is in $E_4^{7,2k-5}$ which vanishes
since by induction hypothesis $E_4^{s,t}$  for $t\leq n-3$
can only  live if $(s,t)$ is of the form $(4m,3)$.
\item
If $r\geq 5$ then $d_r(\alpha A_{k-1})$ is in $E_r^{3+r,2k-r+1}$ which vanishes
since by induction hypothesis $E_r^{s,t}$  for $t\leq n+1-r$
can only live if  $(s,t)=(0,0)$.
\end{itemize}
If  $\alpha A_{k-1}$ is not hit by $d_3$ it survives to $E_\infty$ since
the higher $d_r$ originate from some $E_r^{s,t}$ with $s<0$.
Therefore there must be some $A_k\in E_3^{0,2k}\subset E_2^{0,2k}$
such that $d_3A_k=\alpha A_{k-1}$.
Again if $E_2^{0,2k}$ contained some $\xi$ independent from $A_k$ then it would 
survive\footnote{Note $d_4\xi\in E_4^{4,2k-3}$ can not be $\beta B\in E_4^{4,3}$
since $d_4(\beta B)=\beta^2\not=0$.} to $E_\infty$.
This determines $E_2^{0,n}$ and thus all $E_r^{s,t}$ with $t\leq n+r-2$,
and it is easily checked that these behave as stated.
\item
$n=2k+1\geq 5$.\\
Consider the element $A_{k-1}B\in E_2^{0,2k+1}$.
We have $d_2(A_{k-1}B)=0$ since $d_2A_{k-1}=0$ and $d_2B=0$.
Moreover $d_3(A_{k-1}B)=\alpha A_{k-2}B\not=0$, which shows that $A_{k-1}B\not=0$.\\
Suppose that also $\xi\in E_2^{0,2k+1}$.
Then $d_2\xi=0$ and we may assume that $d_3\xi=0$.
Moreover $d_4\xi$ is in $E_4^{4,2k-2}$ which vanishes
since by induction hypothesis $E_4^{s,t}$  for $t\leq n-3$
can only  live if $(s,t)$ is of the form $(4m,3)$.
Similarly $d_r\xi=0$ for $r\geq 5$, so $\xi$ survives to $E_\infty$.
Thus we see that $A_{k-1}B$ is a basis for $E_2^{0,2k+1}$.
\end{itemize}
Finally we have to prove the divided power structure.
By induction one has
\begin{equation}
\begin{split}
d_3(A_kA_m)
&=d_3(A_k)A_m+A_kd_3(A_m)
=\alpha A_{k-1}A_m+A_k\alpha A_{m-1}\\
&=\alpha\binom{k-1+m}{k-1}A_{k-1+m}+\alpha\binom{k+m-1}{k}A_{k+m-1}\\
&=\alpha\binom{k+m}{k}A_{k+m-1}
=\binom{k+m}{k}d_3A_{k+m}
\end{split}
\end{equation}
and since $d_3$ is injective on $E_3^{0,2k+2m}$ this proves the statement.
\end{proof}

\subsection{The cohomology of $M$ in the dihedral case.}
Now we consider the spectral sequence of
$\xymatrix{M\ar[r]^\pi&BX\ar[r]^\gamma&L}$.
First a general remark about Serre spectral sequences.
\begin{remark}
Let
$\xymatrix{X\ar[r]^j&Y\ar[r]^p&Z}$
be a fibration sequence, with $\pi_1(Z)$ acting trivially
on the cohomology of $X$.
Then there is a filtration
\begin{equation}
H^n(Y;\FF)=F^n_0\supset F^n_1\supset\dots\supset F^n_n\supset (0)
\end{equation}
such that the associated Serre spectral sequence $\{E_r^{s,t}\}$ has
\begin{equation}
E^{s,n}_\infty=F^n_s/F^n_{s+1}
\quad\text{and}\quad
E^{s,t}_2=H^s(Z;H^t(X;\FF))
\end{equation}
See \cite {hatcher2} theorem 1.14.
The map $j^*\colon H^n(Y;\FF)\to H^n(X;\FF)$
can be identified with the composition
$F^n_0\to F^n_0/F^n_1=E^{0,n}_\infty\to E^{0,n}_2$.
Now suppose that $j^*$ is injective.
Then $F^n_s=0$ for $s>0$.
In particular $E^{s,t}_\infty=0$ for $s>0$.
This situation is almost as nice as in the case of a contractible
total space in the sense that few classes survive to $E_\infty$.
In this situation $F^n_0$ can be identified with $E^{0,n}_\infty$
and thus $j^*$ can be identified with the inclusion
$E^{0,n}_\infty\to E^{0,n}_2$.
\end{remark}

\begin{remark}
The situation in the above remark applies here since $\pi^*$ is injective.
As mentioned before we write $P\colon H^n(M;\FF)\to H^{n+1}(M;\FF)$ for the composition 
$\pi^*\psi$ where $\psi$ is as in proposition \ref{shift}.
Thus the image of $P$ can be identified with the image of $\pi^*$,
so with $E_\infty$, which consists of the classes which are in $\ker(d_r)$ for all $r$.
\end{remark}

\begin{theorem}
\label{basis}
The cohomology of $M$ has a basis consisting
of elements $A_m B^eP^jx$
where $m\geq0$ and $e\in\{0,1\}$,
and where either $j>0$ and $x$ runs through a similar basis in smaller dimension,
or $j=0$ and $x=\one$.
\end{theorem}

\begin{proof}
The claim about the basis is inspired by remark \ref{fibo}. 
The proof is modeled on the proof of proposition \ref{cohol}.
We claim that this spectral sequence has the following structure:
\begin{itemize}
\item
The cohomology of $M$ is as stated.
\item
$d_2A_k=0$ and $d_2B=0$.
\item
$d_3A_k=A_{k-1}\alpha$
and $d_3B=0$.
\item
$d_4B=\beta$.
\item
$d_r=0$ for $r>4$.
\item
$d_rP=0$ for all $r$.
\end{itemize}
In other words 
\begin{itemize}
\item
$E_2$ and $E_3$ have a basis of elements
$\alpha^f\beta^g A_mB^eP^jx\in E_2^{4g+3f,2m+3e+j+|x|}$ 
with $k,e,j,x$ as above and $m\geq 0$ and $f\in\{0,1\}$.
\item
$E_4$ has a basis of elements
$\beta^gB^eP^jx\in E_4^{4m,3e+j+|x|}$.
\item
$E_r$ has basis of elements $P^jx\in E_r^{0,j+|x|}$ for $r\geq 5$. 
\end{itemize}
We use induction.
The induction hypothesis $H(n)$ says 
that $E_r^{s,t}$ is as stated for $t+r-2\leq n$.
\begin{itemize}
\item
$n=1$.\\
With $\one\in H^0(M)$ corresponds $\psi(\one)\in H^1(BX)=E^{0,1}_\infty$
and thus $P(\one)\in E^{0,1}_2$.
Suppose that  $E_2^{0,1}$ contained an independent element $\xi$.
Then it would survive to $E_3$ since $d_2\xi\in E_2^{2,0}=0$.
And it would survive to $E_\infty$ since the higher $d_r$ 
point to some $E_r^{s,t}$ with $t<0$.
This would give a $E^{0,1}_\infty=H^1(BX)$ and thus a $H^0(M)$ which is too large.
We conclude that $H^1(M)$ is generated by $P(\one)$.
\item
$n=2$.\\
The element $\alpha\in E_2^{3,0}$ is not hit by $d_2$ because $E_2^{1,1}=0$ since $H^1(L)=0$.
Suppose that it is not hit by $d_3$, and survives to $E^4$.
Then it survives to $E_\infty$ since the higher $d_r$ originate 
from some $E_r^{s,t}$ with $s<0$.
Therefore there must be some $A_1\in E_3^{0,2}\subset E_2^{0,2}=H^2(M)$ such that $d_3A_1=\alpha$.
Secondly $P(\one)\in H^1(M)$ corresponds to $\psi P(\one)\in H^2(BX)=E^{0,2}_\infty$
and thus $P^2(\one)\in E^{0,2}_2$.
If $E_2^{0,2}$ contained some $\xi$ independent from $A_1$ 
then $\xi$ would survive to $E_\infty$.
This would give a $E^{0,2}_\infty=H^2(BX)$ and thus a $H^1(M)$ which is too large.
We conclude that $H^2(M)$ is generated by $A_1$ and $P^2(\one)$.
\item
$n=3$. Part 1.\\
The element $\beta\in E_2^{4,0}$ is not hit by $d_3$ or $d_4$ since $H^1(L)=0=H^2(L)$.
If it were not hit by $d_4$ it would survive to $E_\infty$.
Therefore there must be some $B\in E_4^{0,3}\subset E_2^{0,3}=H^3(M)$ such that $d_4B=\beta$.
Secondly $A_1P(\one)$ is in $E^{3,0}_2$.
Moreover $A_1$ and $P^2(\one)$ in $H^2(M)$ give rise to $P(A_1)$ and $P^3(\one)$ in $E_2^{0,3}$.
By applying $d_3$ and $d_4$ one sees that these $4$ elements are independent.
\item
$n=3$. Part 2.\\
Suppose $z\in E^{3,0}_2$.
\begin{itemize}
\item
$d_2z=0$.
\item
$d_3z$ is in $E^{3,1}_3$ and thus a multiple of $\alpha P(\one)=d_3(AP(\one))$
so by susbstracting $AP(\one)$  from $z$ if necessary we may assume $d_3z=0$
\item
$d_4z$ is in $E^{4,0}_4$ and thus a multiple of $\beta=d_4B$
so by substracting $B$ from $z$ if necessary we may assume that $d_4z=-0$.
\item
$d_rz=0$ for $r\geq 5$.
\end{itemize}
Thus $z$ survives to $E_\infty$ and therefore is in the image of $P$,
consisting of $P(A_1)$ and $P^3(\one)$.
We conclude that   $\{B,A_1P(\one),P(A_1),P^3(\one)\}$ is a basis of  $H^3(M)$.
Note that $B^2=0$ for degree reasons.
\item
$n\geq 4$. Part 1.\\
Consider an element $y=\alpha A_{m-1}B^eP^jx\in E_2^{3,n-2}$, with $m\geq1$.
We assume that either $j>0$ or that $j=0$ and $x=1$.
\begin{itemize}
\item
$y$ can not be hit by $d_2$ since $E^{1,t}_2=0$ for any $t$.
\item
$d_2y=0$ since it is in  $E^{5,n-3}_2$ and $E^{5,t}_2$ vanishes for $t\leq n-1$.
\item
We suppose for the moment that $y$ is not hit by $d_3$.
\item
$d_3y=0$ since it is in $E^{6,n-4}_3$ since  $E^{6,t}_3$ vanishes for all $t$
because  $E^{6,t}_2$ does.
\item
$y$ can not be hit by any $d_r$ with $r>3$ since it would originate in $E^{s,t}_r$ with $s=3-r<0$.
\item
$d_4y=0$ since it is in $E^{7,n-5}_4$ and $E^{s,t}_4$ vanishes for $t\leq n-3$ unless $s\equiv0\mod4$.
\item
$d_ry=0$ for $r\geq 5$ since it is in $E^{3+r,n-r-1}_r$ and $E^{s,t}_r$ vanishes
for $t+r-2\leq n-1$ for such $r$ unless $s=0$.
\end{itemize}
Thus $y$ survives to $E_\infty$ which is a contradiction unless it is hit by $d_3$.
Indeed  if $2(m-1)<n$ then $A_m$ has been introduced before this stage,
and $y$ is the image of $A_mB^eP^jx$ which proves that these elements are independent.
On the other hand if $2(m-1)=n$ and thus $e=j=0$ and $x=\one$ this says 
there must exist some $A_m\in E^{0,n}_3\subset E^{0,n}_2=H^n(M)$ such that $d_3A_m=A_{m-1}$.
\item
$n\geq 4$. Part 2.\\
Consider an element $y=\beta P^jx\in E^{4,n-3}$, again with $x=1$ if $j=0$.
Then $d_ry=0$ for all $r$ since $d_r\beta=0$ and $d_rP^j1=0$ for all $r$.
It can not be hit by $d_r$ for $r\not=4$ since $E^{s,t}_r$ vanishes for $s\leq2s\not=0$ for all $r$.
Therefore it must be hit by $d_4$ and indeed it is the image under $d_4$ of $BP^jx$  which proves
that these elements are independent of each other and of the elements constructed in  Part 1.
\item
$n\geq 4$. Part 3.\\
Finally $E^{0,n}_2$ must contain $E^{0,n}_\infty$ which is the image of $E^{0,n-1}_2$ under $P$.
Applying the induction hypothesis to $E^{0,n-1}_2$ yields that this image
has a basis consisting of the $P^j(x)$.
\item
$n\geq 4$. Part 4.\\
We have shown that $E^{0,n}_2$ contains at least as many independent elements as stated.
We must show that it does not contain anything more.
Thus let $z\in E^{0,n}_2$.
\begin{itemize}
\item
$d_2z=0$ since it is in $E^{2,n-1}_2$ and $E^{2,t}_2$ vanishes for all $t$.
\item
$d_3z$ is in $E^{3,n-2}_3$ which by induction hypothesis has a basis
consisting of elements $\alpha A_kB^eP^jx$.
But these elements are the image under $d_3$ of $A_{k+1}B^eP^jx$.
So by substracting appropriate elements from $z$ we may assume that $d_3z=0$.
\item
$d_4z$ is in $E^{4,n-3}_4$ which by induction hypothesis has a basis
consisting of elements $\beta B^eP^jx$.
The elements with $e=1$ cannot occur because they are mapped
by $d_4$ to independent elements $\beta^2P^jx$.
The elements with $e=0$ are the image under $d_4$ of $BP^jx$.
So by substracting appropriate elements from $z$ we may assume that $d_4z=0$.
\item
$d_rz$ is in $E^{r,n-r+1}_r$  which vanishes by induction hypothesis for $r\geq 5$.
\end{itemize}
This means that $z$ survives to $E_\infty$ and therefore is in the image of $P$,
and thus is a combination of the listed basis elements.
\end{itemize}
Every time that we have found a basis of $E^{0,n}_2$ for some $n$ 
this proves that $E^{s,t}_2$ has the required structure for $t\leq n$
and from this it follows easily that $E^{s,t}_r$ has the required structure
for $t\leq n+2-r$.
\end{proof}

\begin{remark}
From now one we write $A$ for $A_1$.
The element $A_k$ is only defined up to an element of $\ker(d_3)$, 
so up to an element of the form $B^ePx$.
For this reason we can not prove  at this point than they form a  system of divided powers.
However at least for $k<p$ one can force it to be the case  by taking $A_k=\frac{1}{k!}A^k$.
\end{remark}

\begin {remark}
There is a commutative diagram of fibration sequences
\begin{equation}
\xymatrix{
\Omega L\ar[r]\ar[d]^\omega&PL\ar[r]\ar[d]&L\ar[d]\\
M\ar[r]&BX\ar[r]&L\\
}
\end{equation}
and this leads to a map of the Serre spectral sequences by naturality
of the spectral sequence construction.
The image under $\omega^*$ of the elements $A_m$ in $H^{2m}(M)$
satisfy the same recursive relation $d_3A_m=\alpha A_{m-1}$ as
the elements of the same name in $H^{2m}(\Omega L)$.
This implies that $A_m$ maps under $\omega^*$ to the element of the same name.
In particular $\omega^*$ is surjective.
Thus proposition \ref{leray} is indeed applicable.
\end{remark}

\section{The product structure.}

\subsection{Choosing the right $A$.}

From theorem \ref{basis} its is clear that $A$ is only defined up to 
a multiple of $P^2(1)$.
\begin{remark}
There are two ways to measure the contribution of $P^2(1)$.
\begin{itemize}
\item
There is $\lambda_1$ such that
$DA=\lambda_1 P\one$.
\item
There is $\lambda_2$ such that
$\mu A=A\otimes\one+\lambda_2 P\one\otimes P\one+\one\otimes A$.
\end{itemize}
If we add $\nu P^2\one$ to $A$ then both $\lambda_i$ change by $\nu$.
\end{remark}
Both ways amount to the same:
\begin{proposition}
$\lambda_2=A(y;y,y)=\lambda_1$.
\end{proposition}
\begin{proof}
As noted in remark \ref{iota} the generator $\iota_X$
of $H_0(X;X)$ can be represented by $(y;)$ for any $y\in X$. 
Now $2D\iota_X$ is represented by $2(y;y)$ in $H_1(X;X)$
and by  $(1;y)+(\eta(y);y)$ in $H_1(G;X)$.
From  
$(P(\one ))(D\iota_X))=(DP(\one ))(\iota_X)=\one (\iota_X)=1$
one sees that $D\iota_X$ generates $H_1(X;X)$.
Now
\begin{equation}
\begin{split}
2\lambda_2
&=2(\mu A)(D\iota_X\otimes D\iota_X)
=2A(\mu(D\iota_X\otimes D\iota_X))\\
&=A(\mu((y;y)\otimes ((1;y)+(\eta(y);y)))=2A(y;y,y)\\
\lambda_1&=(DA)(D\iota_X)=A(D^2\iota_X)
=A(y;y,y)
\end{split}
\end{equation}
\end{proof}
We now choose $A$ in such a way that the $\lambda_i$ vanish.
So $DA=0$ and $A$ is primitive, and $A(y;y,y)=0$ for all $y\in X$.
This has a nice consequence:
\begin{proposition}
\label{aisq}
If $A$ is chosen as above then
$\psi A\in H^3_R(X)$ is in the image of $H^3_Q(X)$.
Thus $\psi A$ must be equivalent to the class of
\cite{mochizuki1}. 
\end{proposition}
\begin{proof}
From $\partial(y;a,y,y)=(y;y,y)-(y*a;y,y)$
and the fact that $X=R_p$ is connected on sees
that $(z;y,y)$ is homologous to $(y;y,y)$ for all $y,z$.\\
Since the cohomology class of $DA$ vanishes
there exists a cochain $F$ such that $\partial^*F=DA$.
Now let $G(y,z)=F(y)$ and $A'=A+\partial^*G$ then
\begin{equation}
\begin{split}
&(\partial G)(y;y,z)
=G(\partial(y;y;z))
=G((y;y)-(y*z;y*z))\\
&=F(y)-F(y*z)
=F(\partial(y;z))
=(\partial F)(y;z)\\
&=(DA)(y;z)
=-A(y;y,z)
\end{split}
\end{equation}
Therefore $A'(y;y,z)=A(y;y,z)+(\partial G)(y;y,z)=0$.
\end{proof}

\subsection{Choosing the right $B$.}

In this paper we will use the notation $\Delta$ for the Bockstein homomorphism
$H_{n+1}(C;\FF)\to H_n(C;\FF)$,
which is defined for any chain complex $C$ of free abelian groups,
and is natural for chain maps.

\begin{remark}
If we identify $H_n(C\otimes C';\FF)$ with $\sum  H_k(C;\FF)\otimes H_m(C';\FF)$
then $\Delta$ corresponds with $\Delta\otimes 1+(-1)^k 1\otimes\Delta$.\\
In line with our convention in \ref{notat} we write also $\Delta$ for the dual map,
which differs  from the Bockstein homomorphism
$H^n(C;\FF)\to H^{n+1}(C;\FF)$ by a sign.
If we identify $H^n(C\otimes C';\FF)$ with $\sum  H^k(C;\FF)\otimes H^m(C';\FF)$
then $\Delta$ corresponds with
$1\otimes\Delta+(-1)^m\Delta\otimes 1$.
\end{remark}

The Bockstein homomorphism anticommutes 
with the boundary operator in the long exact sequence
associated to an exact sequence of chain complexes.
Therefore it anticommutes with transgression.
Since $A\in H^2(M;\FF)$ transgresses to $\alpha\in H^3(BG;\FF)$,
its image $\Delta A$ transgresses to $-\Delta\alpha=\beta$.
This shows that $\Delta A$ is a suitable choice for $B$.
\begin{proposition}
We have $\mu^*(B)=B\otimes\one+\one\otimes B$.
\end{proposition}

\begin{proof}
Writing $p\colon M\times M\to M$ for the projection on the first factor
we have
$\Delta(A\otimes\one)=\Delta p^*A=p^*\Delta A=p^*B=B\otimes\one$
and similarly $\Delta(\one\otimes A)=\one\otimes B$,
so $\mu^*(B)=\mu^*(\Delta A)=\Delta\mu^*(A)=\Delta(A\otimes\one+\one\otimes A)
=B\otimes\one+\one\otimes B$.
\end{proof}

\begin{remark}
Another argument for the primitivity of $B$ could go along the following lines.
The multipication map $\mu$ provides a map of fibrations
from the fibration $M\times M\to B(X)\times M\to B(G)$ 
to the fibration $M\to B(X)\to B(G)$.
By naturality of the Serre spectral sequence this 
provides a map of spectral sequences.
From this one sees that $\mu^*B$ can
not have a contribution $A\otimes P\one$.
Thus the most general expression for
$\mu^*B$ is
\begin{equation}
\mu^*B=B\otimes\one+\lambda_3 P^2\one\otimes P\one
+\lambda_4 P\one\otimes P^2\one+\lambda_5 P\one\otimes A+\one\otimes B
\end{equation}
Moreover from the associativity of $\mu$ one can deduce that $\lambda_3=\lambda_4$.
If we add $\nu_2 P^3\one+\nu_3 PA$ to $B$ the effect is
adding $\nu_2$ to $\lambda_3$ and $\lambda_4$ and adding $\nu_3$ to $\lambda_5$.
Thus we can force all $\lambda_i$ to vanish.
\end{remark}

\begin{proposition}
\label{bisq}
If $B$ is chosen as above then
$\psi B\in H^4_R(X)$ is in the image of $H^4_Q(X)$.
\end{proposition}

\begin{proof}
This follows from proposition \ref{aisq}, the fact $B=\Delta A$,  and
the fact that that the diagram
\begin{equation}
\xymatrix{
H^3_Q(X)\ar[r]\ar[d]^\Delta&H^3_R(X)\ar[d]^\Delta&H^2_R(X;X)\ar[l]_\psi\ar[d]^\Delta\\
H^4_Q(X)\ar[r]&H^3_R(X)&H^3_R(X;X)\ar[l]_\psi\\
}\end{equation}
is commutative because the $\psi$ and the canonical map
from the rack complex to the quandle complex are both chain maps.
\end{proof}

\subsection{The system of divided powers.}

For the next proposition we note that $d_2$ vanishes and thus
$d_3$ is defined on every $x\in H^n(B(G;X))$.
\begin{proposition}
\label{d3prim}
If $d_3x=0$ and $\mu x=x\otimes\one+\one\otimes x$ then $x=0$ for $n\geq4$.
\end{proposition}

\begin{proof}
By the proof of theorem \ref{basis} the elements $x\in H^n(M)$
in the kernel of $d_3$ are precisely the elements of the form
$x=Py+BPz$ with  $y\in H^{n-1}(M)$ and $z\in H^{n-4}(M)$.
We can write
\begin{equation}
\mu y=y\otimes\one+\one\otimes y+\sum_i a_i\otimes b_i,\qquad
\mu z=z\otimes\one+\one\otimes z+\sum_j u_j\otimes v_j
\end{equation}
for certain elements $a_i,b_i,u_j,v_j$ of positive dimension.
Now we have
\begin{equation}
\begin{split}
&\mu x - x\otimes\one -\one\otimes x\\
&\qquad=P\one\otimes y
-P\one\otimes Bz
+B\otimes Pz
+BP\one\otimes z
+(-1)^{n-1}Pz\otimes B\\
&\qquad\qquad
+\sum Pu_i\otimes v_i
+\sum BPu_j\otimes v_j
+\sum (-1)^{k_j+1} Pu_j\otimes Bv_j
\end{split}
\end{equation}
where $k_j$ is the dimension of $v_j$.

In particular the only contribution of the form $B\otimes\dots$ is $B\otimes Pz$.
Thus if $\mu x - x\otimes\one -\one\otimes x=0$ then $Pz=0$.
But now the only contribution of the form $P\one\otimes\dots$ is $P\one\otimes y$.
Thus if $\mu x - x\otimes\one -\one\otimes x=0$ then $y=0$.
\end{proof}

\begin{proposition}
If $x$ is a cohomology class of positive dimension then $x^p=0$.
\end{proposition}

\begin{proof}
If $x$ has odd dimension then $x^2=0$, so we may assume that $x$ has even dimension $n$.
We can write
\begin{equation}
\mu x=x\otimes\one+\one\otimes x+\sum_i u_i\otimes v_i
\end{equation}
for some elements $u_i,v_i$ of positive dimensions smaller than $n$. 
Since $n$ is even all terms in the above sum commute and we have
\begin{equation}
\mu(x^p)=(\mu x)^p=x^p\otimes\one+\one\otimes x^p+\sum_i (u_i\otimes v_i)^p
\end{equation}
By induction hypothesis we have $u_i^p=0$ and $v_i^p=0$.
Thus we conclude that  $\mu(x^p)=x^p\otimes\one+\one\otimes x^p$.
Moreover $d_3x$ is defined and $d_3$ is a derivation
so $d_3(x^p)=px^{p-1}d_3x=0$.
Therefore $x^p$ satisfies the hypothesis of proposition \ref{d3prim}
and so vanishes. 
\end{proof}

\begin{remark}
One can easily deduce using induction from the Rota-Baxter formula in remark \ref{rota2} that
\begin{equation}
(Px)^n=nP((Px)^{n-1}x)
\end{equation}
for every $x$ of odd degree.
In particular $(Px)^p=0$ for all $x$,
in line with the above proposition.
\end{remark}

\begin{theorem}
The generators $A_m$ can be chosen
in such a way that they form a system of divided powers.
\end{theorem}

\begin{proof}
We use induction.
The induction hypothesis $H(e)$ claims
that the $A_m$ are defined for $m<p^e$
in such a way that $d_3A_m=\alpha A_{m-1}$ for $m<p^e$
and such that
\begin{equation}
A_{m_1}A_{m_2}=
\begin{cases}
\binom{m_1+m_2}{m_1}A_{m_1+m_2}&\text{ if }m_1+m_2<p^e\\
0&\text{ if } m_1+m_2\geq p^e
\end{cases}
\end{equation}
For $e=1$ this accomplished by choosing $A_m=\frac{1}{m!}(A)^m$.

For the induction step the proof of theorem \ref{basis} shows that there can be chosen 
some $A_{p^e}$ such that $d_3 A_{p^e}=A_{p^e-1}$.
Now  we define
\begin{equation}
A_{p^eq+r}=\frac{((p^e)!)^qr!}{(p^eq+r)!}(A_{p^e})^q A_r
\end{equation}
for $1\leq q<p$ and $0\leq r<p^e$.
Note that the numerical factor is the inverse of an integer
which is nonzero modulo $p$.
It can easily be checked that $A_m$ 
so defined satisfy $H(e+1)$.
\end{proof}

\begin{remark}
As far as the cup product  structure is concerned this only means that the 
algebra generated by the $A_m$ is a polynomial algebra
generated by  the $A_{p^e}$ with as only relations $(A_{p^e})^p=0$.
However we will see in the next section that also for the 
coproduct $\mu$ it is worthwile to think secretly
of $A_m$ as $\frac{1}{m!}A^m$.
\end{remark}

\section{The coproduct structure.}

\subsection{Fixing generators for $H_n(M;\FF)$ for small $n$.}

We fix the following notations:
\begin{itemize}
\item
We write $r$ for the element $D\iota_X\in H_1(M;\FF)$ which satisfies $\langle P\one,r\rangle=1$.
\item
Let $s\in H_2(M;\FF)$ be the element such that
$\langle A,s\rangle=1$
and $\langle P^2\one,s\rangle=0$.
\item
Let $t\in H_3(M;\FF)$ be the element such that
$\langle B,t\rangle=1$,
$\langle AP\one,t\rangle=0$,
$\langle PA,t\rangle=0$ and
$\langle P^3\one,t\rangle=0$.
\end{itemize}
From this follows that $\Delta t=s$.

\begin{proposition}
One has $\mu(s\otimes t)=\mu(t\otimes s)$.
\end{proposition}

\begin{proof}
We must check that $s\otimes t-t\otimes s$ vanishes on $\mu x$
when $x$ runs through the basis elements  of $H^5(M;\FF)$,
which are
$APA$, 
$AP^2\one$,
$BP^2\one$,
$A_2 P\one$,
$A B$,
$PAP^2\one$,
$PBP\one$,
$PA_2$,
$P^2 AP\one$,
$P^2 B$,
$P^3A$ and $P^5\one$
according to theorem \ref{basis}.
Here we can assume $A_2$ to be $\frac{1}{2} A^2$.

One can evaluate $\mu x$ from the formulas
for $\mu A_1$, $\mu B$ and $\mu\circ P$.
For example
\begin{equation}
\begin{split}
&\mu(APA)
=\mu(A)\mu(PA)\\
&=(A\otimes\one+\one\otimes A)((P\otimes\one)((A\otimes\one+\one\otimes A)+\one\otimes PA)\\
&=APA\otimes\one+AP\one\otimes A+A\otimes PA
+PA\otimes A+P\one\otimes 2A_2+\one\otimes APA\\
\end{split}
\end{equation}
What we must check is that in each case the contribution
of $B\otimes A$ is the same as the contribution of $A\otimes B$.
In fact only $\mu(AB)$ contains any of these terms.
\end{proof}

\begin{proposition}
If $p>3$ then $t^2=0$.
\end{proposition}
\begin{proof}
We must check that $t\otimes t$ vanishes on $\mu x$
when $x$ runs through the basis elements  of $H^6(M;\FF)$,
which are
$APAP\one$,
$APB$,
$AP^2A$,
$AP^4\one$,
$BPA$,
$BP^3\one$,
$A_2P^2\one$,
$ABP\one$,
$A_3$,
$PAPA$,
$PAP^2\one$,
$PBP^2\one$,
$PA_2P\one$,
$PAB$,
$P^2AP^2\one$,
$P^2BP\one$,
$P^2A_2$,
$P^3AP\one$,
$P^3B$,
$P^4A$ and
$P^6\one$.
However only for $p>3$ can we assume assume $A_3$ to be $\frac{1}{6} A^3$.
What we must check is that in each case the contribution
of $B\otimes B$ vanishes.
This is easily checked in all cases for $p>3$.
For $p=3$ we do not know $\mu A_3$.
\end{proof}

From $r,s,t$ we can thus form expressions
$s^{m_1}t^{e_1}r^{j_1}s^{m_2}t^{e_2}r^{j_2}\dots$
with $s_i,r_i\in\{1,2,\dots\}$ and $e_i\in\{0,1\}$.
We will show that these are in fact a basis for the homology
by evaluating them on our cohomology basis.

\subsection{A partial order on monomials.}

We want to describe a partial order on the
cohomology basis $x_i$ and on the homology monomials $a_j$
with the property that the matrix formed by
the values $\langle x_i,a_j\rangle$ is a triangular matrix.
 
We define a \emph{node} $N$ to be a finite sequence consisting of the symbols $R,S,T$.
There is an obvious notion of concatenation of nodes.
\begin{itemize}
\item
The replacement of a node $N_1(S,R) N_2$
by the node $N_1(R,S) N_2$
is called an elementary reduction.
\item
Likewise the replacement of a node $N_1(T,R) N_2$
by the node $N_1(R,T) N_2$
is called an elementary reduction.
\item
The replacement of a node $N_1(S,T) N_2$
by a the node $N_1(T,S) N_2$
is called an elementary equivalence.
Similarly for the inverse operation.
\end{itemize}
We write $N\geq N'$ if we can get  from $N$ to $N'$ 
by a sequence of elementary reductions and elementary equivalences.
We write $N\sim N'$ if we can get  from $N$ to $N'$ 
by a sequence of elementary equivalences.

To any node $N$ we associate its $S$-value $v_S(N)$
by counting how may times a symbol $S$
is on the left of a symbol $R$.
The $T$-value $v_T(N)$ is defined similarly.
Obviously for each elementary reduction one of the two
values decreases and the other stays the same,
and for each elementary equivalence both stay the same.
Thus if $N\geq N'$ and $N'\geq N$ then $N\sim N'$.
Clearly we get in this way a partial order on the set 
of equivalence classes under $\sim$.

Note also that two nodes can only be comparable
if they contain the same amount of $R$, of $S$ and of $T$.
From this partially ordered set we eliminate all 
classes represented by a node 
of the form $N_1 (T,T) N_2$.
Furthermore we  call a node \emph{pure} if it contains only the symbol $S$.

To each node class $N$ we associate a cohomology
basis element $x_N$ and a homology element $a_N$
as follows.
If $N$ is empty then $x_N=1$ and $a_N=1$.
If $N=(S)^{m}(T)^{e}(R)^{j} K$
then $x_N=A_m B^e P^j x_K$
and $a_N=s^mt^er^j a_K$.
Thus to a pure node corresponds some $A_m$ in cohomology
and some power of $s$ in homology.

\begin{proposition}
\label{pairing}
Suppose that $\mu A_m=\sum_{j=0}^m A_j\otimes A_{m-j}$ for $2m<n$,
and that $N$ is not pure.
Then  $\langle x_N,a_K\rangle\not=0$ for classes of dimension $n$ implies $N\geq K$.
Moreover $\langle x_N,a_N\rangle=1$ or $-1$.
\end{proposition}

\begin{proof}
We use induction on the size of $N$.
One can write $N=(S)^m(T)^e(R)^jV$ with $j>0$ or with $j=0$ and $V$ empty.
We assume that $j>0$ since the other case is easy.
Then $x_N=A_mB^eP^j x_V$ and therefore
\begin{equation}
\begin{split}
\mu x_N
=\Bigl(\sum_{\substack{k+l=m\\f+g=e}}
A_kB^f\otimes A_lB^g\Bigr)
\Bigl( (P^j\otimes 1)\mu x_V+\sum_{u+v=j}P^u 1\otimes P^v x_V\Bigr)
\end{split}
\end{equation}
We consider three cases:
\begin{itemize}
\item
Suppose that $K=(S)W$ for some node $W$.
Then $a_K=sa_W=\mu(s\otimes a_W)$,
so $\langle x_N,a_K\rangle=\langle\mu x_N,s\otimes a_W\rangle$.
This means that we must look for terms of the form $A\otimes z$ in $\mu x_N$.

The term with $(P^j\otimes 1)\mu x_V$ can give no contribution.
The sum over $u$ can only give a contribution for $f=0$ and $u=0$ and $k=1$,
and in particular $m\geq 1$.
The contributing term is $A_1\otimes A_{m-1}B^eP^jx_V$, 
which can be written as $A\otimes x_U$, 
where $U=(S)^{m-1}(T)^e(R)^jV$.

The contribution is $\langle A\otimes x_U,s\otimes a_W\rangle=\langle x_U,a_W\rangle$.
By induction hypothesis this can only be nonvanishing if $U\geq W$.
But then $N=(S)U\geq(S)W=K$ as desired.
\item
Suppose that $K=(T)W$ for some node $W$.
This case is completely analogous to the first one.
\item
Suppose that $K=(R)W$ for some node $W$.
Then $a_K=ra_W=\mu(r\otimes a_W)$,
so $\langle x_N,a_K\rangle=\langle\mu x_N,r\otimes a_W\rangle$.
This means that we must look for terms of the form $P\one\otimes z$ in $\mu x_N$.
There are two subcases.
\begin{itemize}
\item
The term with $(P^j\otimes 1)\mu x_V$ can only give a contribution if $j=1$ and $k=0$ and $f=0$.
Only the term $x_V\otimes 1$ of $\mu x_V$ can be involved.
The contributing term is $P\one\otimes A_mB^ex_V$,
which can be written as $P\one\otimes x_U$,
where $U=(S)^m(T)^eV$.
The contribution is  $\langle P\one\otimes x_U,r\otimes a_W\rangle=\pm\langle x_U,a_W\rangle$.
By induction hypothesis this can only be nonvanishing if $U\geq W$.
But then 
\begin{equation*}
N=(S)^m(T)^e(R)V\geq (R)(S)^m(T)^eV
=(R)U\geq (R)W=K
\end{equation*}
as desired.
\item
The sum over $u$ can only give a contribution for $f=0$ and $u=1$ and $k=0$.
The contributing term is $P\one\otimes A_mB^eP^{j-1}x_V$, 
which can be written as $P\one\otimes x_U$, 
where $U=(S)^m(T)^e(R)^{j-1}V$.
The contribution is  $\langle P\one\otimes x_U,r\otimes a_W\rangle=\pm\langle x_U,a_W\rangle$.
By induction hypothesis this can only be nonvanishing if $U\geq W$.
But then 
\begin{equation*}
N=(S)^m(T)^e(R)^jV\geq (R)(S)^m(T)^eR^{j-1}V
=(R)U\geq (R)W=K
\end{equation*}
as desired.
\end{itemize}
\end{itemize}
It can easily be checked that $\langle x_N,a_N\rangle$ is in fact 
$(-1)^{\ell(\ell-1)/2}$, where $\ell$ is the number of occurences of  $R$ or $T$ in $N$.
\end{proof}

\begin{proposition}
Assume that the $a_K$ form a basis of homology in dimension $<d$,
and that $\mu A_m=\sum_j A_j\otimes A_{m-j}$ for $2m<d$.
For $d=2n$ also assume that $s^n\not=0$.
Then the $a_K$ form a basis in homology in dimension $d$,
and if $d=2n$ we can find $A_n$ such that  $\mu A_n=\sum_{j=0}^n A_j\otimes A_{n-j}$,
\end{proposition}

\begin{proof}
Suppose that there is a linear relation between
the elements $a_K$ associated to impure nodes $K$.
Then by applying the elements $x_N$
associated to impure nodes $N$ we see
that all coefficients must vanish.

Now suppose that $s^n$ is a linear combination
of elements $a_N$.
Then again by applying the $x_N$ associated to impure
nodes $N$ we see that all coefficients vanish,
and we get the contradiction $s^n=0$.
Therefore $s^n$ and the $a_K$ are independent,
and since their number is equal to the number
of basis elements $x_N$ they are a basis of $H_{2n}(M)$.

Now we define $A_n$ as the class which is $1$ on $s^n$
and $0$ on the $a_K$ associated to impure nodes.
Consider the expression 
\begin{equation}
\begin{split}
&\langle \mu A_n -\sum_{j=1}^{n-1} A_j\otimes A_{n-j},a_K\otimes a_L\rangle\\
&\qquad=\langle A_n,a_Ka_L\rangle-\sum_j\langle A_j,a_K\rangle\cdot\langle A_{n-j},a_L\rangle\\
\end{split}
\end{equation}
By assumption $A_n$ vanishes on $a_{KL}=a_Ka_L$ if $KL$ is an impure node,
which is the case unless $K$ and $L$ are both pure.
On the other hand $\langle A_j,a_K\rangle$ vanishes unless $K=(S)^j$
and $\langle A_{n-j},a_L\rangle$ vanishes unless $L=(S)^{n-j}$.

We see that the above expression vanishes for all $K$ and $L$.
Since the $a_K$ and the $a_L$ form a basis,
this proves that
$\mu A_n -\sum_{j=1}^{n-1} A_j\otimes A_{n-j}$
has no contribution other than in dimensions $(0,2n)$ and $(2n,0)$.
But the the contribution in these dimensions
are obviously $A_n\otimes1$ and $1\otimes A_n$.
\end{proof}

\begin{remark}
Suppose we can prove that $s^n\not=0$ in the critical degrees,
which are the powers of $p$.
Then the above proposition says that $r,s,t$ generate homology
and that $st=ts$ and $t^2=0$ are the only relations between them.
Equivalently the $A_m$ satsify $\mu A_m=\sum A_k\otimes A_{m-k}$
for all $m$, which means that $A_m$ behaves
as if it were $\frac{1}{m!}A^m$.
In the next subsection we will prove that indeed $s^{p^e}\not=0$.
\end{remark}

\begin{proposition}
\label{bokan}
$\Delta A_n=A_{n-1}B$
if the $\mu A_m$ formula is satisfied for $m\leq n$.
\end{proposition}

\begin{proof}
Let $K=(S)^m(T)^e(R)^jL$ with $j>0$.
Then
\begin{equation}
\Delta a_K
=\Delta(s^mt^er^ja_L)
=\begin{cases}
s^mr^j\Delta a_L&\text{ if }e=0\\
s^{m+1}r^ja_L+s^mtr^j\Delta a_L&\text{ if }e=1\\
\end{cases}
\end{equation}
Thus we see by induction that $\Delta a_K$
is a sum of classes $a_U$ associated to nodes $U$
starting with $(S)^{m'}(T)^{e'}(R)^{j'}$ with $j'\geq j>0$.
Such a node is incomparable with $(S)^n$ 
and thus evaluates to $0$ on $A_n$.
However if $K=(S)^m(T)^e$ then $a_K=s^mt$ and
\begin{equation}
\Delta a_K
=\Delta(s^mt)
=\begin{cases}
0&\text{ if }e=0\\
s^{m+1}&\text{ if }e=1\\
\end{cases}
\end{equation}
We see that $\langle\Delta A_n,a_K\rangle=\langle A_n,\Delta a_K\rangle$
is nonvanishing only if $K=(S)^{n-1}(T)$.
But the class $x_K=A_{n-1}B$ associated to $K=(S)^{n-1}(T)$ 
is characterized by this property.
\end{proof}

\subsection{The Thomas operation and its use.}

We use the following notations:
\begin{itemize}
\item
$\phi$ is the map on cohomology 
induced by map $\ZZ/(p)\to\ZZ/(p^2)$ on coefficients
given by multiplication by $p$.
\item
$\eta$ is the map on cohomology
induced by map $\ZZ/(p^2)\to\ZZ/(p)$ on coefficients
given by projection.
\item
If $\psi$ is a cohomology operation $H^n\to H^m$ 
then $\sigma\psi$ is the composition
$H^{n-1}(X)\cong H^n(\Sigma X)\to H^m(\Sigma X)\cong H^{m-1}(X)$
using the suspension $\Sigma$.
\item
As before $\Delta\colon H^n(X;\ZZ/(p))\to H^{n+1}(X;\ZZ/(p))$
is the Bockstein operator associated to the coefficient sequence
$0\to \ZZ/(p)\to\ZZ/(p^2)\to\ZZ(/(p)\to 0$.
\item
$\DDelta\colon H^n(X;\ZZ/(p^2))\to H^{n+1}(X;\ZZ/(p))$
is the Bockstein operator associated to the coefficient sequence
$0\to \ZZ/(p)\to\ZZ/(p^3)\to\ZZ/(p^2)\to 0$.
\end{itemize}
It can easily be seen that $\DDelta\circ \phi=\Delta$.
We cite the following theorem from \cite{browdert}:
\begin{proposition}
There exists a cohomology operation $C\colon H^{2n}(X;\ZZ/(p))\to H^{2pn}(X;\ZZ/(p^2))$
with the following properties:
\begin{itemize}
\item
$\eta C(u)=u^p$
and $C\eta(u)=u^p$.
\item
\begin{equation}
C(u_1+u_2)=C(u_1)+C(u_2)+\phi\Bigl(\sum_{i=1}^{p-1}\frac{1}{p}\binom{p}{i}
u_1^i\cup u_2^{p-i}\Bigr)
\end{equation}
\item
$\sigma C=0$.
\end{itemize}
Moreover these properties determine $C$ uniquely.
\end{proposition}

We need the following additional fact about $C$:
\begin{proposition}
$\DDelta Cx=x^{p-1}\Delta x$.
\end{proposition}

\begin{proof}
Consider the operation $\psi$ defined by
$\psi(x)=\DDelta Cx-x^{p-1}\Delta x$.
Then $\sigma\psi=0$ since $\sigma C=0$ and $\sigma$ anticommutes with $\DDelta$
and since cup products in a suspension vanish.

The following argument is an adaptation of the proof in \cite{browdert}
proving the uniqueness of $C$.
By the description in \cite{cartan} of the algebra of cohomology operations
any operation can be split uniquely as a sum of two parts:
\begin{itemize}
\item
The first part is  a composition of Bockstein operations and Pontrjagin operations.
On this part $\sigma$ is injective.
\item
The second part consists of operations which are decomposable, 
viewed as elements in the cohomology of an Eilenberg-MacLane space.
On this part $\sigma$ vanishes since cup products in a suspension vanish.
\end{itemize}
From this we see that $\psi$ is decomposable.

From $\DDelta\circ \phi=\Delta$ and the fact that $\Delta$ is a derivation
one checks easily that
\begin{equation}
\begin{split}
&\DDelta\phi\Bigl(\sum_{i=1}^{p-1}\frac{1}{p}\binom{p}{i}u_1^i\cup u_2^{p-i}\Bigr)\\
&\qquad=(u_1+u_2)^{p-1}\Delta(u_1+u_2)-u_1^{p-1}\Delta u_1+u_2^{p-1}\Delta u_2
\end{split}
\end{equation}
This means that the operation $\psi$ is additive,
which means that its is primitive,
viewed as an element in the cohomology of an Eilenberg-MacLane space;
see theorem 5.8.3 in \cite{whiteheadg}.
Thus $\psi$ is decomposable and primitive and of odd degree.
By proposition 4.23 of \cite{milnorm} this implies that $\psi$ vanishes.
\end{proof}

\begin{theorem}
Let $q=p^e$,
and assume that the $\mu A_m$ formula is satsified for $m<pq=p^{e+1}$.
Then $s^{pq}\not=0$.
\end{theorem}

\begin{proof}
Consider the following commutative diagram:
\begin{equation}
\xymatrix{
&H^{2q}(M;\ZZ/(p))\ar[d]^C\\
H^{2pq}(M,\ZZ/(p))\ar[r]^\phi\ar[dr]^\Delta&H^{2pq}(M;\ZZ/(p^2))\ar[r]^\eta\ar[d]^\DDelta&H^{2pq}(M;\ZZ/(p))\\
&H^{2pq+1}(M;\ZZ/(p))\\
}
\end{equation}
Since $\eta C(A_q)=A_q^p=0$ we can choose some $\CALA\in H^{2pq}(M,\ZZ/(p))$
such that $\phi\CALA=C A_q$.
We have
\begin{equation}
\Delta\CALA
=\DDelta\phi\CALA
=\DDelta C A_q
=A_q^{p-1}\Delta A_q
=A_q^{p-1}A_{q-1}B
=c A_{pq-1}B
\end{equation}
where $c=\frac{(pq-1)!}{(q!)^{p-1}(q-1)!}$ is nonzero modulo $p$.
Therefore
\begin{equation}
\langle\CALA, s^{pq}\rangle=\langle\CALA,\Delta(s^{pq-1}t)\rangle
=\langle\Delta\CALA,s^{pq-1}t\rangle=c\not=0
\end{equation}
Thus $\CALA$ detects $s^{pq}$.
\end{proof}

This completes the proof that the algebra structure and coalgebra structure
on the cohomology of $M$ are as stated.
\begin{remark}
In the yet undecided case $p=3$ this still proves that
$s^3$ complements $\{s^2r^2,\dots,r^6\}$ to a basis.
Thus there is $c\in\FF$ such that $t^2-cs^3$ is a combination
of $s^2r^2,\dots,r^6$.
By applying $x_N$ with $N$ impure we see that all coefficient must vanish.
So at least we have $t^2=cs^3$ for some $c$.
Equivalently $\mu A_3=A_3\otimes1+A_2\otimes A+A\otimes A_2+1\otimes A_3-cB\otimes B$.

If $t^2=s^3$ the replacement of a node $N_1(T,T)N_2$ by a node
$N_1(S,S,S)N_2$ should be added as an elementary equivalence.
Then $2v_S+3v_T$ decreases for an elementary reduction 
and stays the same for an elementary equivalence.
So we still get a useful partial order.
\end{remark}

\section{Integral homology and cohomology.}

Let $X$ be any quandle,
and choose a base point $y\in X$.
Then this choice defines a map from the
one point rack to $X$.
On the other hand there is unique
map from $X$ to the one point rack.
Together these two maps split the rack complex
of $X$ into the rack complex of the one point rack 
and a complemantary summand. 
The homology of the first part is infinite cyclic
in each dimension $n$, generated by $r^n$.

Now we concentrate on the dihedral case $X=R_p$.
The homology of the complementary part 
is generated by the monomials other than $r^n$.
We will prove that the homology is 
$p$-torsion, by checking that 
the kernel of the Bockstein operator $\Delta$
acting on this part equals the image of $\Delta$.
This settles one of the conjectures in \cite{niebp2}.

Given a chain complex $C$ the notation $C[i]$ 
stands for the shifted complex given by
$C[i]_j=C_{i+j}$.
Obviously $H_j(C[i])$ is canonically isomorphic to $H_j(C)[i]$.

Now let $Z$ denote the chain complex with basis 
the $a_N$ associated to nodes $N$ which are not 
of the form $R^{n}$, with the Bockstein operator $\Delta$ as boundary operator.
Furthermore let  $Y$ denote the subcomplex
with basis the $s^mt^e$.
\begin{proposition}
The chain complex $Y$ is acyclic.
\end{proposition}

\begin{proof}
Obvious since
$\Delta(s^mt)=s^{m+1}$ for $m\geq 0$ and
$\Delta{s^m}=0$ for $m\geq 1$.
\end{proof}

\begin{proposition}
The chain complex $Z$ is acyclic.
\end{proposition}

\begin{proof}
There is an isomorphism of chain complexes
\begin{equation}
\sigma\colon Y\oplus\bigoplus_{j>0}\Bigl(Y[j]\otimes X\Bigr)\to Z
\end{equation}
which on $Y$ is the inclusion, 
and which on $Y[j]\otimes Z$ is given by
\begin{equation}
\sigma(s^mt^e\otimes a_K)=s^mt^jr^ja_K
\end{equation}
as the following computations show:
\begin{equation}
\begin{split}
&\Delta_Z\sigma(s^mt\otimes a_K)
=\Delta_Z(s^mtr^ja_K)\\
&=s^{m+1}r^ja_K+(-1)^{j+1}s^mtr^j\Delta_Z a_K\\
&=\sigma(s^{m+1}t\otimes a_K+(-1)^{j+1}s^mt\otimes\Delta_Z a_K)\\
&=\sigma(\Delta_Y\otimes 1+(-1)^j(-1)^{|s^mt|}1\otimes\Delta_Z)(s^mt\otimes a_K)\\
\end{split}
\end{equation}
\begin{equation}
\begin{split}
&\Delta_Z\sigma(s^m\otimes a_K)
=\Delta_Z(s^mr^ja_K)=(-1)^{j}s^mr^j\Delta_Z a_K\\
&=\sigma(s^{m}t\otimes a_K+(-1)^{j}s^mt\otimes\Delta_Z a_K)\\
&=\sigma(\Delta_Y\otimes 1+(-1)^j(-1)^{|s^m|}1\otimes\Delta_Z)(s^mt\otimes a_K)\\
\end{split}
\end{equation}
Therefore by the Kunneth theorem we have
\begin{equation}
H(Z)\cong H(Y)\oplus \bigoplus_{j>0} \Bigl(H(Y)[j]\otimes H(Z)\Bigr)
\end{equation}
and since $H(Y)$ is trivial, so is $H(Z)$.
\end{proof}

\section{Quandle homology and cohomology.}

\subsection{Constructing quandle cocycles.}

We will call an element $x\in H^n(X;X;\FF)$ a quandle class
iff $\psi(x)\in H^{n+1}(X;\FF)$ is a quandle class.
Thus $A$ and $B$ are quandle classes.
We write $H^n_Q(X;X;\FF)$ for the subgroup
of $H^n(X,X;\FF)$ consisting of quandle classes.

\begin{proposition}
If $F$ and $G$ are quandle classes then so is $F\cup G$.
\end{proposition}

\begin{proof}
Let $F$ be represented by a cocycle $f\in C^k(X;X;\FF)$
such that $\psi f$ vanishes on degenerate elements of $\BB(X)_{k+1}$,
and let $G$ be represented by a cocycle $g\in C^m(X;X;\in\FF)$ 
such that $\psi g$ vanishes on degenerate elements of $\BB(X)_{m+1}$.
We will show that $\psi(f\cup g)$ vanishes on any degenerate element
$x=(x_0,x_1,x_2,\dots,x_{k+m})$ by evaluating $f\cup g$
on $y=\psi_\bullet(x)=(x_0;x_1,\dots,x_{k+m})$.
By construction of the cup product we have
\begin{equation}
\begin{split}
(f\cup g)(y)
&=(-1)^{km}\sum_A \epsilon(A)\cdot f(\delta^0_A(y))\cdot g(\delta^1_B(y))
\end{split}
\end{equation}
We consider the ways in which $x$ can be degenerate.
\begin{itemize}
\item
Suppose that $x_0=x_1$.
\begin{itemize}
\item
If $1\in B$ then  $\psi_\bullet^{-1}\delta^0_A(y))$ is degenerate
so $\psi f$ vanishes on it.
\item
If $1\in A$ then  $\psi_\bullet^{-1}\delta^1_B(y))$ is degenerate
so $\psi g$ vanishes on it.
\end{itemize}
\item
Suppose that $x_i=x_{i+1}$ for some $i\geq 1$.
\begin{itemize}
\item
If $\{i,i+1\}\subset B$ then  $\psi_\bullet^{-1}\delta^0_A(y))$ is degenerate
so $\psi f$ vanishes on it.
\item
If $\{i,i+1\}\subset A$ then  $\psi_\bullet^{-1}\delta^1_B(y))$ is degenerate
so $\psi g$ vanishes on it.
\item
Suppose that $A=U\cup\{i\}$ and $B=V\cup\{i+1\}$.
Then there is a companion term associated to $K=U\cup\{i+1\}$ and $L=V\cup\{i\}$.
Since $\delta^0_Ay=\delta^0_Ky$ and $\delta^1_By=\delta^1_Ly$
and $\epsilon(K)=-\epsilon(A)$ their contributions  cancel.
\end{itemize}
\end{itemize}
\end{proof}

\begin{proposition}
If $F\in H^n(X;X;\FF)$ is a quandle class then  so is $QF$.
\end{proposition}

\begin{proof}
If for $x=(x_1,\dots,x_n)\in X^n$ one has $x_j=x_{j+1}$ for some $j$
then the same is true for $\partial^0_i x$ unless $i=j$ or $i=j+1$,
but $\partial^0_jx$ and $\partial^0_{j+1}$ give opposite contributions to $\partial^0 x$.
Thus $\partial^0$ maps the degeneracy subcomplex of $C^R_\bullet(X)$ to itself.
So if $F\in H^{n+1}(X;\FF)$ is a quandle class then so is $\partial^0F\in H^{n+2}(X;\FF)$.
The claim follows since the operator $Q$ on $H^n(X;X;\FF)$ 
corresponds under $\psi$ with the operator $\partial^0$ on $H^{n+1}(X;X;\FF)$
by  proposition \ref{qisdel}.
\end{proof}

Now we concentrate again on the dihedral case.
\begin{proposition}
The multiplicative generators $A_{p^n}$ can be chosen
to be quandle classes.
\end{proposition}

\begin{proof}
We use induction in $n$.
We know that $A_1$ is a quandle class.
Now assume that $A_{p^f}$ is a quandle class for $f<n$,
and therefore also their product $A_{p^n1}$.
Assume also that $\Delta A_{p^f}=A_{p^f-1}B$ for $f<n$.
From the last section we know that all torsion in rack cohomology is of order $p$.
The same must be true of quandle cohomology since it is a direct
summand of rack cohomology.
Thus on $H^*_Q(X;X;\FF)$ the kernel of $\Delta$
coincides with the image of $\Delta$.
In particular we can choose a quandle class
$\CALA$ such that $\Delta \CALA=A_{p^n-1}B$.

In order to show that $\CALA$ is a valid choice for $A_{p^n}$ 
we must check that it is not in the algebra generated
by lower dimensional generators.
So assume that $\CALA$ is a linear combination of 
basis elements $x_N$ associated to impure nodes $N$.
Then $\Delta \CALA=A_{p^n-1}B$ is a sum of terms $\Delta x_N$.
However if $N=(S)^m(T)^e(R)^jK$ with $j>0$ for some node $K$ then
\begin{equation}
\Delta x_N=\Delta(A_mB^eP^jx_K)
=A_{m-1}B^{e+1}P^jx_K+A_mB^eP^j\Delta x_K
\end{equation}
which is a sum of basis elements each associated 
with a node with the same amount of $R$.
So these terms can give no contribution
to $A_{p^n-1}B$ wich itself is a basis element
associated to $(S)^{p^n-1}T$.
The only other possibility $N=(S)^m(T)$ has  $\Delta x_N=0$
\end{proof}

\subsection{Independence of quandle cohomology classes.}

This is also about the dihedral case.
We call a node $N$ a Q-node if does not contain
two consecutive symbols $R$ and does not
end with an $R$.
To each Q-node class we associate a quandle
cohomology class $y_N$ as follows.
If $N=(S)^m(T)^e$ then $y_N=x_N=A_mB^e$.
If $N=(S)^m(T)^e(R)K$ then $y_N=A_mB^eQ y_K$.
We will show that these classes are linearly 
independent.

\begin{proposition}
$x_N\cup P(\one)$ is a combination of elements $x_K$ with $K>(R)N$.
\end{proposition}

\begin{proof}
We use induction on the size of $N$.
\begin{itemize}
\item
It is true for $N=(S)^m(T)^e$ since
\begin{equation}
x_{(S)^m(T)^e}\cup P(\one)=A_mB^eP(\one)=x_{(S)^m(T)^e(R)}
\end{equation}
and $(S)^m(T)^e(R)>(R)(S)^m(T)^e$.
\item
Suppose  that it is true for $N$ and  write
$x_N\cup P(\one)=\sum_K  c_K x_K$ 
with $K>(R)N$ and $c_K\in\FF$ then
\begin{equation}
\begin{split}
&x_{{S}^m(T)^e(R)N}\cup P(\one)
=A_mB^eP(x_N)\cup P(\one)\\
&=A_mB^e(\pm P^2(x_N)+P(x_N\cup P(\one)))\\
&=\pm x_{(S)^m(T)^e(R)^2N}+\sum_K c_K x_{(S)^m(T)^e(R)K}
\end{split}
\end{equation}
where $(S)^m(T)^e(R)K>(S)^m(T)^e(R)^2N>(R)(S)^m(T)^e(R)N$.
Therefore it is true for $(S)^m(T)^e(R)N$.
\end{itemize}
\end{proof}

\begin{proposition}
If $N$ is a Q-node then
$y_N-x_N$ is a combination of elements $x_K$ with $K>N$.
\end{proposition}

\begin{proof}
It is trivially true for $N=(S)^m(T)^e$.
Suppose that it is true for $N$ and write
$y_N=x_N+\sum_K c_K x_K$ with $K>N$.
By the preceding proposition we may write
\begin{equation}
\begin{split}
&x_N\cup P(\one)=\sum_{U>(R)N}c'_U x_U\\
&x_K\cup P(\one)=\sum_{V>(R)K} c''_{K,V} x_V\\
\end{split}
\end{equation}
With $\epsilon=(-1)^{|y_N|}$ we now have
\begin{equation}
\begin{split}
y_{(S)^m(T)^e(R)N}
&=A_mB^eQ(y_N)
=A_mB^e(P(y_N)+\epsilon y_n\cup P(\one))\\
&=A_mB^eP(x_N)+A_mB^e\sum_K c_KP(x_K)\\
&\qquad+\epsilon A_mB^ex_N\cup P(\one)
+\epsilon A_mB^e\sum_K  x_K\cup P(\one)\\
&=x_{(S)^m(T)^e(R)N}
+\sum_K c_K x_{(S)^m(T)^e(R)K}\\
&\qquad+\sum_U \epsilon c'_U x_{(S)^m(T)^eU}
+\sum_K\sum_V \epsilon c''_{K,V} x_{(S)^m(T)^eV}
\end{split}
\end{equation}
where $(S)^m(T)^eV>(S)^m(T)^e(R)K>(S)^m(T)^e(R)N$
and also $(S)^m(T)^eU>(S)^m(T)^e(R)N$.
Thus it is true for $(S)^m(T)^e(R)N$.
\end{proof}

\begin{proposition}
The quandle classes $y_N$ associated to Q-nodes $N$ are independent.
\end{proposition}

\begin{proof}
Suppose that some nontrivial linear combination $\sum c_N y_N$ vanishes.
Then among  the $N$ for which $c_N\not=0$ there is one which
is minimal for the partial order on nodes.
But then it is the only term of the sum which gives  anontrivial 
contribution to $x_N$, a contradiction.
\end{proof}

Thus the rank of a quandle cohomology group is at least
as large as the number of Q-nodes contributing to that dimension.
We will see shortly that we have in fact equality.

\subsection{Generating quandle homology.}

Suppose that we are in the situation of  proposition \ref{bxxisbgx}.
In particular a base point $y\in X$ has been chosen.
Then $(y)$ is a cycle and defines an element $\rho$ of $H_1^R(X;\FF)$,
independent of the choice of $y$.
The map from $H_n(G;X;\FF)$ to $H_{n+1}(X;\FF)$
which maps $c$ to $\mu(\rho\otimes c)$ 
coincides with the composition 
\begin{equation}
H_n(G;X;\FF)\to H_n(X;X;\FF)\to H_{n+1}(X;\FF)
\end{equation}
of the isomorphisms $\psi$ and $\chi$.
In other words $H_*(X;\FF)$ is a free module over $H*(G;X;\FF)$
with one generator $\rho$.

Now we specialize to the dihedral case.
In that case $H_*(G;X;\FF)$ is the algebra generated by $r$, $s$ and $t$.
The above remark shows that we get all of quandle homology
by letting this algebra act from the right on $\rho$.
Moreover the fact that $2r$ is represented by 
$(1;y)+(\eta(y);y)$  implies the following:
\begin{itemize}
\item
$\rho r$ vanishes since $2\rho r$ is represented
by $(y)\bigl((1;y)+(\eta(y);y)\bigr)=2(y,y)$.
\item
If $\theta$ is any rack homology class then $\theta r^2$
vanishes in quandle homology.
In fact $4\theta r^2$ is a sum of terms ending in $\dots,y,y)$. 
\end{itemize}
To each node $N$ we associate an element $b_N$ of quandle homology
as follows.
If $N$ is empty then $b_N=\rho$.
If $N=(S)^m(T)^e(R)^jK$ then
$b_N=b_K r^jt^es^m$.
The remarks above prove that the $b_N$ generate the quandle homology groups.
They show also that any $b_N$ vanishes unless $N$ is a Q-node.

Thus the rank of a quandle homology group is at most as large
as the number of Q-nodes contributing to that dimension.
However the rank of the homology group and the rank of the cohomology group
are the same.
Therefore the inequality in the last subsection and
the one in this subsection must both be equalities.
This proves the `delayed Fibonacci sequence' conjecture in \cite{niebp2}.

\vfill\eject

\end{document}